\documentclass[a4paper,12pt]{article}

\usepackage[a4paper, margin=2.5cm]{geometry}

\usepackage{amsthm,latexsym,MnSymbol,amsmath,enumerate,float}

\newtheorem{lemma}{Lemma}
\newtheorem{proposition}{Proposition}
\newtheorem{theorem}{Theorem}
\newtheorem{corollary}{Corollary}

\newtheorem{case}{Case}

\newtheorem{example}{Example}
\newtheorem{remark}{Remark}

\def\nd{\nmid}

\newcommand{\bmaut}{\mathop{\mathrm{Aut}}}
\newcommand{\bminn}{\mathop{\mathrm{Inn}}}

\begin{document}

\begin{center}

\LARGE \textbf{Minimal embeddings of small finite groups}

\end{center}

\begin{center}

\large Robert Heffernan, Des MacHale and Brendan McCann

\end{center}

\begin{abstract}

We determine the groups of minimal order in which all groups of order $n$ can embedded for $ 1 \leq n \leq 15$. We further determine the order of a minimal group in which all groups or order $n$ or less can be embedded, also for $ 1 \leq n \leq 15 $.


\end{abstract}

\section{Introduction}

The question as to whether a given group, or collection of groups, can be embedded in another plays an important role in group theory.           A very early result in this direction is Cayley's Theorem (Cayley \cite{Cayley1}, \cite{Cayley2}), which states that any group of order $n$ can be expressed as a permutation group on $n$ symbols. Equivalently, all groups of order $n$ or less can be embedded in $S_n$, the symmetric group of degree $n$. 				Embeddings in symmetric groups of minimal degree have been studied by, amongst others, Johnson \cite{Johnson}, Wright \cite{Wright}  and, more recently, Saunders (\cite{Saunders1}, \cite{Saunders2}). However embeddings in (abstract) groups of minimal order do not appear to have been given much consideration thus far.          For example, one might ask for what values of $n$ is the order of $S_n$ minimal with respect to the embedding of all groups of order $n$ or less.                     A more general problem is to determine a group, or groups, of minimal order in which a given finite collection $\Gamma$ of finite groups can be embedded.          This is likely to be difficult when $ | \Gamma | $ is large, especially where the groups in $\Gamma$ do not have coprime orders.  

The purpose of the present paper is consider the case where $ \Gamma $ comprises either the groups of order $n$, or the groups of order $n$ or less, for small values of $n$. Thus, in Sections \ref{order8} - \ref{p-odd}, we determine the groups of smallest order in which all groups of order $n$ can be embedded for $ n = 1, \dots, 15$, and find the minimal order of a group in which all groups of order $p^3$ can be embedded, where $p$ is an odd prime. Then, in Sections \ref{soluble}, \ref{non-soluble} and \ref{leqnleq15} we consider the more complex problem of finding examples of groups of minimal order in which all groups of order $n$ or less can be embedded, for $ n = 1, \dots, 15 $. Apart from the trivial cases $ n = 1, 2 $, it will be shown that the latter groups are not unique. Furthermore, for $ 6 \leq n \leq 15 $, the orders of the respective minimal groups will be shown to be less than $ n! = |S_n| $.

A useful elementary lower bound for the order of such minimal groups is found by noting that if the $p$-group $P$ has a cyclic subgroup $P_1$,  of order $p^{k_1}$, and an elementary abelian subgroup $P_2$, of rank $k_2$, then $ | P_1 \cap P_2 | \leq p $, so
\[ |P| \geq |P_1P_2| = \frac{|P_1||P_2|}{|P_1 \cap P_2|} \geq \frac{p^{k_1}p^{k_2}}{p} = p^{k_1 + k_2 -1} .\]
It follows that if all groups of order $p^k$ can be embedded in $G$, then the Sylow $p$-subgroups of $G$ will have order at least $p^{2k-1}$. This yields the following lower bound:

\begin{lemma}\label{pbound}

Let $p$ be a prime and let $G$ be a finite group in which all groups of order $p^k$ can be embedded. Then $|G|$ is a multiple of $p^{2k - 1}$.
\end{lemma}

More generally we have

\begin{lemma}\label{nbound}

Let $ \{ p_1, \dots , p_m \} $ denote the set of distinct primes that are less than or equal to $n$. For $ i = 1, \dots , m $ let $ k_i$ be maximal such that $ p_i^{k_i} \leq n $. If $G$ is a finite group in which all groups of order $n$ or less can be embedded, then $|G|$ is a multiple of $ p_1^{2k_1 - 1} \cdots \ p_m^{2k_m - 1} $.

\end{lemma}

Thus, for example, Lemma \ref{pbound} shows that if $G$ is a group of minimal order in which all groups of order $8$ can be embedded, then $|G|$ is a multiple of 32, while Lemma \ref{nbound} shows that if all groups or order 12 or less can be embedded in $G$, then $|G|$ is a multiple of $ 2^5 \cdot 3^3 \cdot 5 \cdot 7 \cdot 11 = 332 \ 640 $. In Section \ref{order8} we will show that the former bound is attained, whereas the results of Sections \ref{soluble} and \ref{non-soluble} will show that the latter is not. In general, we note that the bound given by Lemma \ref{nbound} will be attainable only for a relatively small number of values of $n$. Indeed Lemma  \ref{p^knotembed}, in Section \ref{p-odd} below, shows that when $p$ is an odd prime and $ k \geq 3 $, the lower bound given by Lemma \ref{pbound} cannot be attained. Furthermore, the same result shows that the bound given by Lemma \ref{nbound} will always be exceed for $n\geq 27$.

In what follows our notation will generally be standard and all groups, and collections of groups, will be finite. For clarity we note that $D_n$ will denote the dihedral group of order $2n$, given by $ \langle  \ a, \ b \mid  a^{n} = 1 = b^{2}, \ bab = a^{-1} \ \rangle$, and $Q_{n}$ will denote the dicyclic group of order $4n$, given by $ \langle \ a, \ b \mid  a^{2n} = 1, \  b^{2} = a^{n}, \ b^{-1}ab = a^{-1} \ \rangle$. In particular $Q_{2}$ is the quaternion group of order 8 and $ Q_{3} $ is the non-trivial extension of $ C_{3} $ by $ C_{4} $.

We note that some exploratory work was done with the aid of the computational group theory system GAP~\cite{GAP}.

\section{Embeddings of groups of order 8}\label{order8}

In this section we will show that $G$ is a group of minimal order in which all groups of order 8 can be embedded if and only if $G$ is isomorphic either to the semi-direct product of a cyclic group of order 8 with its automorphism group, or to the direct product of a quasi-dihedral group of order 16 with a cyclic group of order 2. We first describe certain groups in which not all groups of order 8 can be embedded.

\begin{lemma}\label{Habex4}

Let $G$ be a group such that $ 2^5 \top |G| $ and suppose that $G$ has a subgroup $ H $ such that
\begin{enumerate}[(i)]

\item
$ |H| = 16 $;
\item
$H$ is abelian;
\item
$H$ has exponent at most $4$.

\end{enumerate}
Then not all groups of order $8$ can be embedded in $G$.

\end{lemma}

\begin{proof}

By Sylow's Theorems we may assume that $ |G| = 32 $. We suppose that $C_{8}$ can be embedded in $G$ and let $ U \leqslant G $ be such that $ U = \langle x \rangle \cong C_{8} $. Now $| G : H | = 2 $ and $H$ has exponent at most $4$ so, by comparison of orders, $ G = HU $ and $ H \cap U = \langle x^{2} \rangle \cong C_{4} $. As $H$ and $U$ are both abelian, we have $ \langle x^{2} \rangle \leqslant Z(G) $. We suppose further that $ C_{2} \times  C_{2} \times  C_{2} $ can also be embedded in $G$ and let $ V  \leqslant G $ be such that $ V = \langle  v_{1}, \ v_{2}, \ v_{3} \rangle \cong C_{2} \times  C_{2} \times  C_{2} $. Since $ \langle x^{2} \rangle \leqslant Z(G) $ we see that $ V \langle x^{2} \rangle $ is an abelian group. 

If $ V \cap \langle x^{2} \rangle = 1$ then $ |V \langle x^{2} \rangle| = 32 $ so $G$ is abelian and the non-abelian groups of order $8$ cannot be embedded in $G$. Hence we may suppose that $ |V \cap \langle x^{2} \rangle| = 2 $. Without loss of generality we let $  V \cap \langle x^{2} \rangle =  \langle x^{4} \rangle = \langle v_{3} \rangle$. Then $ V \langle x^{2} \rangle = \langle v_{1} \rangle \times \langle v_{2} \rangle \times \langle x^{2} \rangle \cong C_{2} \times C_{2} \times C_{4} $. Now $ | G : V \langle x^{2} \rangle | = 2 $ so $  V \langle x^{2} \rangle \unlhd G $. But  $ V = \langle g \in V \langle x^2 \rangle \mid o(g) = 2 \rangle $  is a characteristic subgroup of $ V \langle x^{2} \rangle $ and is thus also normal in $G$. Now if $Q_{2}$ can be embedded in $G$ then we let $ Q \leqslant G $ be such that $ Q \cong Q_{2}$ and see, by comparison of orders, that $ Q \cap V = Z(Q) \cong C_{2}$ and $ G = QV $. But then $ G/V = QV/V \cong Q/(Q \cap V) = Q/Z(Q) \cong C_{2} \times C_{2} $ so $G$ has exponent $4$ (since $V$ has exponent $2$), and a contradiction arises. We thus conclude that not all groups of order $8$ can be embedded in $G$.
\end{proof}

We now define two groups of order 32 in which all groups of order 8 can be embedded. We note that the groups of order 8 are: $ C_{2} \times C_{2}\times C_{2}, \ C_{2} \times C_{4}, \ C_{8}, \ D_{4}$ and $Q_{2} $. We let $ \langle x_1 \rangle \cong C_2 $ and let $ H_1 $ be the quasi-dihedral group of order 16, defined by
\[ H_1 = \langle \ x_2, \ y \mid x_2^2 = y^8 = 1 , \  x_2yx_2 = y^3 \ \rangle . \] 
We form the direct product $ \langle x_1 \rangle \times H_1 \cong C_2 \times H_1 $. Thus $ | \langle x_1 \rangle  \times H_1| = 32 $. We can verify that $ \langle \ x_1, \ x_2, \ y^4 \ \rangle  \cong  C_2 \times C_2 \times C_2 $, $\langle \ x_1, \ y^2 \rangle  \ \cong  C_2 \times C_4 $ and $ \langle y \rangle  \cong  C_8 $. We further see that $ (y^2)^{x_2} = (y^{x_2})^2 = y^6 = y^{-2} $. Thus $ \langle \ x_2, \ y^2 \ \rangle \cong D_4 $. In addition we have $ (x_2y)^2 = ( x_2yx_2y ) = y^{x_2}y = y^3y = y^4 $, and $ (y^2)^{x_2y} = (y^{-2})^y = y^{-2}$. Hence $ \langle \ y^2, \ x_2y \ \rangle $ has order 8, is non-abelian and is generated by elements of order 4. We conclude that $ \langle \ y^2, \ x_2y \ \rangle \cong Q_2 $. Therefore all groups of order 8 can be embedded in $C_2 \times H_1$.

We now let 
\[ H_2 = \langle \ x_1, \ x_2, \ y \mid  x_1^2 = x_2^2 = y^8 = 1 , \ [ x_1, x_2 ] = 1 , \ x_1yx_1 = y^5 , \  x_2yx_2 = y^3 \ \rangle . \]   
Thus $H_2$ is isomorphic to the holomorph of $C_8$, that is the semi-direct product of $ \langle y \rangle \cong C_8 $ by $ \bmaut( C_8) $ (identified with $ \langle \ x_1, \ x_2 \ \rangle $). As above, we can confirm that $ \langle \ x_1, \ x_2, \ y^4 \ \rangle  \cong  C_2 \times C_2 \times C_2 $, $\langle \ x_1, \ y^2 \ \rangle  \cong  C_2 \times C_4 $, $ \langle y \rangle  \cong  C_8 $, $ \langle \ x_2, \ y^2 \ \rangle \cong D_4 $ and $ \langle \ y^2, \ x_2y \ \rangle \cong Q_2 $. Thus $C_2 \times H_1$ and $H_2$ are non-isomorphic groups of order 32 in which all groups of order 8 can be embedded.

From Lemma \ref{pbound} we see that $C_2 \times H_1$ and $H_2$ are groups of minimal order in which all groups of order 8 can be embedded. Our next result shows that, up to isomorphism, they are the only such groups.

\begin{theorem}
If $G$ is a group of order 32 in which all groups of order 8 can be embedded then $G$ is isomorphic either to $C_2 \times H_1$ or $H_2$, as defined above.
\end{theorem}

\begin{proof}
Let $G$ be a group of order 32 in which all groups of order 8 can be embedded. Then $G$ has a subgroup $ \langle y \rangle$ such that $ \langle y \rangle \cong C_8$. In addition $G$ has a subgroup $\langle \ x_1,\  x_2, \ x_3 \ \rangle$ with $\langle \ x_1, \ x_2, \ x_3 \ \rangle \cong C_2 \times C_2 \times C_2 $. Since $|G| = 32$ we have $\langle \ x_1, \ x_2, \ x_3 \rangle \cap \langle y \rangle = \langle y^4 \rangle \cong C_2 $. Without loss of generality we let $ y^4 = x_3 $ and see that $G$ is the product $ G = \langle \ x_1, \ x_2 \ \rangle \langle y \rangle $. Since $G$ is nilpotent $ \langle y \rangle $ is a proper subgroup of $ N_G(\langle y \rangle) $, so we may assume that $ \langle x_1 \rangle \leqslant N_G(\langle y \rangle) $. Hence
\[ \langle y \rangle \unlhd \langle \ x_1,\  y \ \rangle \unlhd \langle \  x_1, \ x_2 , \ y \ \rangle = G .\]

Now $ o(x_1) = 2 $ and $ x_1 \notin \langle y \rangle $ (since otherwise $ G = \langle y \rangle \langle x_2 \rangle $ and $|G| = 16$). Therefore $ \langle \ x_1, \ y \ \rangle $ is not cyclic. We have $ \langle y^2 \rangle = \Phi(\langle y \rangle) \leqslant \Phi( \langle \ x_1, \ y \ \rangle ) $ and $ | \langle \ x_1, \ y \ \rangle : \Phi( \langle \ x_1, \ y \ \rangle ) | \geq 4 $ (since $\langle \ x_1, \ y \ \rangle $ is not cyclic). Hence, by comparison of orders,
\[ \langle y^2 \rangle = \Phi( \langle \ x_1, \ y \ \rangle ) \unlhd G . \]
Since $ \langle y^2 \rangle \cong C_4 $ and $ \bmaut(C_4) \cong C_2$, we have $ | G : C_G(\langle y^2 \rangle ) | \leq 2$. If $ \langle y^2 \rangle \leqslant Z(G) $ then $ \langle \ x_1, \  x_2 , \ y^2 \ \rangle $ is an abelian subgroup of exponent 4 and index 2 in $G$. But then, by Lemma \ref{Habex4}, not all groups of order 8 can be embedded in $G$. Thus $ \langle y^2 \rangle \not\leqslant Z(G) $, so $ | G : C_G( \langle y^2 \rangle ) | = 2 $. Now $ C_G( \langle y^2 \rangle ) = \langle y \rangle ( C_G( \langle y^2 \rangle ) \cap \langle \ x_1, \ x_2 \ \rangle ) $, so we may assume, without loss of generality, that $ x_1 \in C_G( \langle y^2 \rangle ) $ and $ x_2 \notin C_G( \langle y^2 \rangle) $. In particular we have $ C_G( \langle y^2 \rangle ) = \langle \ x_1, \ y \ \rangle $.

We now proceed on a case by case basis:

\begin{case} 

$ \langle y \rangle \unlhd G $ and $ y^{x_1} = y $.

\end{case}

Here $ x_1 \in Z(G) $ and $ G = \langle x_1 \rangle \times \langle \ x_2, \ y \ \rangle $. Since $G$ is not abelian $x_2$ normalises, but does not centralise, $ \langle y \rangle $. Therefore we have $ y^{x_2} = y^k $, where $ k $ = 3, 5  or 7. If $ y^{x_2} = y^5 $ then $ (y^2)^{x_2} = ( y^{x_2} )^2 = ( y^5 )^2 = y^{10} = y^2 $, which cannot be the case since $x_2 \notin C_G ( \langle y^2 \rangle ) $. If $ y^{x_2} = y^7$ ($= y^{-1}$) then
\[  ( x_2 y^n)^2 = x_2 y^n x_2 y^n  = (y^n)^{x_2} y^n = y^{-n} y^n =1. \]
Hence the only elements of order 4 in $ \langle \ x_2, \  y \ \rangle $ are $ y^2$ and $ y^6$, and it follows that all elements of order 4 in $G$ are contained in $ \langle \ x_1, \  y^2 \ \rangle $, which is abelian. Thus $Q_2$, which has non-commuting elements of order 4, cannot be embedded in $G$. We conclude that $ y^{x_2} = y^3$, and it follows that $ G = \langle x_1 \rangle \times \langle \ x_2,\  y \ \rangle \cong C_2 \times H_1 $.

\begin{case}

$ \langle y \rangle \unlhd G $ and $ y^{x_1} \neq y $.

\end{case}

Since $x_1$ centralises $y^2$, but does not centralise $y$, we must have $ y^{x_1} = y^5 $. Now $ x_2 $ normalises $ \langle y \rangle $, but does not centralise $ \langle y^2 \rangle $. It follows that $ \langle \ x_1, \ x_2 \ \rangle / C_{ \langle  x_1,  x_2  \rangle }( \langle y \rangle ) \cong C_2 \times C_2 \cong \bmaut( C_8 ) \cong \bmaut( \langle y \rangle )$. Hence we see that $ G $ is the semi-direct product of $ \langle y \rangle $ by $ \langle \ x_1, \ x_2 \ \rangle $, so $ G \cong H_2 $. 

\begin{case}\label{Prop1Case3} 

$ \langle y \rangle \not\unlhd G $ and $ y^{x_1} = y $.

\end{case}

In this case we have $ x_1 \in Z(G)$. Furthermore $ \langle \ x_1, \  y \ \rangle / \langle x_1 \rangle \unlhd G / \langle x_1 \rangle $, and $ \langle \ x_1, \ y \ \rangle / \langle x_1 \rangle \cong C_8 $. If $ [ y, x_2 ] \in \langle x_1 \rangle $ then $x_2$ centralises $y^2$, which is ruled out. Thus $ [ y, x_2 ] \notin \langle x_1 \rangle $ and, since $ \langle y \rangle \not\unlhd G$, we must have $ y^{x_2} = y^kx_1 $, where $ k \not\equiv 1 $ mod(8). Furthermore $ \langle \ x_1, \ y \ \rangle / \langle x_1 \rangle = \langle \ y^kx_1 \ \rangle \langle x_1 \rangle / \langle x_1 \rangle \cong C_8 $, so it follows that $  k \equiv $ 3, 5 or 7 mod(8). If $ k \equiv 5 $ mod (8) then
\[ (y^2)^{x_2} = (y^{x_2})^2 = ( y^5x_1)^2 = y^{10}x_1^2 = y^2 . \]
But $x_2$ does not centralise $y^2$, so $ k \equiv 5 $ mod (8) is ruled out. Hence we have either $ y^{x_2} = y^3x_1 $ or $ y^{x_2} = y^7x_1 $ ($ = y^{-1}x_1 $).

We let $ Q \leqslant G $ be such that $ Q \cong Q_2 $. By comparison of orders we have $ \langle y \rangle \cap Q \neq 1 $. Thus $ \langle y^4 \rangle \leqslant \langle y \rangle \cap Q $. But $Q_2$ has a unique subgroup of order 2 so
\[ \langle y^4 \rangle = Z(Q) = Q \cap Z(G). \]
We further see that if $ g \in Q $ is such that $ o(g) = 4 $ then $ g^2 = y^4 $. In addition, since $y^4$ is the unique element of order 2 in $Q$ and $ y^4 \neq x_1 $, we see that $ Q \cap \langle x_1 \rangle = 1 $. Thus $ Q \langle x_1 \rangle / \langle x_1 \rangle \cong Q \cong Q_2 $, so $ G / \langle x_1 \rangle $ has a subgroup isomorphic to $Q_2$. Now, if $ y^{x_2} = y^7x_1 = y^{-1}x_1$ then $ G / \langle x_1 \rangle $ is isomorphic to $ \langle \ x, y \mid x^2 = y^8 = 1, \ xyx = y^{-1} \ \rangle \cong D_8 $. But $D_8$ has a unique subgroup of order 4, so $Q_2$ cannot be embedded in $G / \langle x_1 \rangle$. Thus we cannot have $ y^{x_2} = y^7x_1 $.

If $ y^{x_2} = y^3x_1 $ then
\begin{eqnarray*}
(x_2y^n)^2 & = & x_2y^nx_2y^n\\
& = & (y^n)^{x_2}y^n\\
& = & (y^3x_1)^ny^n\\
& = & y^{3n}x_1^ny^n\\
& = & x_1^ny^{4n}.
\end{eqnarray*}
Hence, if $n$ is odd, we have
\[ (x_2y^n)^2 = x_1y^{4n} \neq y^4 ;\]
and, if  $n =2k$ is even:
\[ (x_2y^n)^2 = x_1^{2k}y^{4(2k)} = (x_1^2)^k (y^8)^k= 1 \neq y^4 .\] 
Since $ x_1 \in Z(G) $ we similarly have
\[ (x_1x_2y^n)^2 = x_1^2(y^nx_2)^2 = (x_2y^n)^2 \neq y^4 .\]
It follows that if $ g \in G $ is such that $ g^2 = y^4 $, then $ g \in \langle \ x_1, \ y \ \rangle $, which is abelian. But, from above, $Q$ ($\cong Q_2$) is generated by elements whose square is equal to $y^4$, so $ Q \leqslant \langle \ x_1, \ y \ \rangle $, and a contradiction results. We conclude that if $ \langle y \rangle \not \unlhd G$ and $ y^{x_1} = y $ then not all groups of order 8 can be embedded in $G$.

\begin{case} $ \langle y \rangle \not\unlhd G $ and $ y^{x_1} \neq y $.

\end{case}

Since $ x_1 $ normalises $ \langle y \rangle $ and centralises $y^2$ we have $ y^{x_1} = y^5 $. In particular $ [ y, x_1 ] = y^{-1}y^{x_1} = y^{-1}y^5 = y^4 \in \langle \ x_1, \ y^2 \ \rangle $. Since $ x_1^{x_2} = x_1 $ and $ \langle  y^2 \rangle = \Phi ( \langle \ x_1 , y \ \rangle ) \unlhd G $ we further see that $ \langle \ x_1, \ y^2 \ \rangle \unlhd G $. In particular, since $ \langle \ x_1, \ y^2 \ \rangle / \langle y^2 \rangle $ is a normal subgroup of order 2 in $ G/ \langle y^2 \rangle $, we have $ \langle \ x_1, \ y^2 \ \rangle / \langle y^2 \rangle \leqslant Z( G / \langle y^2 \rangle )$. Since $ \langle y \rangle \not\unlhd G $ we have $ \langle y \rangle / \langle y^2 \rangle \not\unlhd G / \langle y^2 \rangle $, so $ G / \langle y^2 \rangle $ is a non-abelian group of order 8. As $ \langle \ x_1, \ x_2, \ y^2 \ \rangle / \langle y^2 \rangle \cong \langle \ x_1, \ x_2 \ \rangle \cong C_2 \times C_2 $ it follows that $ G / \langle y^2 \rangle \cong D_4 $.

As in Case \ref{Prop1Case3} we let $ Q $ be a subgroup of $G$ such that $ Q \cong Q_2 $. If $ \langle y^2 \rangle \leqslant Q $ then $ \langle y^2 \rangle \unlhd Q $ and $ Q / \langle y^2 \rangle \cong C_2 $. Since $ \langle \ x_1, \ y^2 \ \rangle / \langle y^2 \rangle = Z ( G / \langle y^2 \rangle ) $, the subgroups of order 2 in $ G / \langle y^2 \rangle $ are: $ \langle \ x_1, \ y^2 \ \rangle / \langle y^2 \rangle $, $ \langle \ x_2, \ y^2 \ \rangle / \langle y^2 \rangle $,  $ \langle \ x_1x_2, \ y^2 \ \rangle / \langle y^2 \rangle $, $ \langle \ x_1y, \ y^2 \ \rangle / \langle y^2 \rangle $ and $ \langle  y \rangle / \langle y^2 \rangle $. But $y^2$ is centralised by $ x_1$, so $ \langle \ x_1, \ y^2 \ \rangle $, $ \langle \ x_1y, \ y^2 \ \rangle $ and $ \langle y \rangle $ are abelian. We can further confirm that $ \langle \ x_2, \ y^2 \ \rangle $ and $ \langle \ x_1x_2, \ y^2 \ \rangle $ are both isomorphic to $D_4$, so $Q$ cannot be identical to any of these subgroups. Therefore $ \langle y^2 \rangle \not\leqslant Q $, whence $ Q \cap \langle y^2 \rangle = Q \cap \langle y \rangle = \langle y^4 \rangle  = Z(Q) $. Then
\[ Q \langle y^2 \rangle / \langle y^2 \rangle \cong Q / ( Q \cap \langle y^2 \rangle ) = Q / Z(Q)  \cong C_2 \times C_2. \]
Now $ \langle \ x_1, \ x_2 , \ y^2 \ \rangle / \langle y^2 \rangle $ and $ \langle \ x_1, \ y \ \rangle / \langle y^2 \rangle $ are the only subgroups isomorphic to $ C_2 \times C_2 $ in $ G / \langle y^2 \rangle $ ($\cong D_4$). Hence either $ Q \leqslant  \langle \ x_1, \ x_2, \ y^2 \ \rangle $ or $ Q \leqslant \langle \ x_1, \ y \ \rangle $. Since $x_1$ centralises both $x_2 $ and $y^2$ we have $ \langle \ x_1, \  x_2, \ y^2 \ \rangle = \langle x_1 \rangle \times \langle \ x_2, \ y^2 \ \rangle \cong C_2 \times D_4 $. But all elements of order 4 in $ C_2 \times D_4 $ commute with each other, so $Q_2$ cannot be embedded in $ \langle \ x_1, \ x_2, \ y^2 \ \rangle $. We further have
\begin{eqnarray*}
(x_1y^n)^2 & = & x_1y^nx_1y^n\\
&= & y^{5n}y^n\\
& = & y^{6n} .
\end{eqnarray*}
Hence $ o( x_1y^n) = 4 $ if and only if $ 2 \top n $. Thus all elements of order 4 in $ \langle \ x_1, \ y \ \rangle $ are contained in $ \langle \ x_1, \ y^2 \ \rangle $, which is abelian. Therefore $Q_2$ cannot be embedded in $ \langle \ x_1, \ y \ \rangle $, and hence in $G$ in this instance.
\end{proof}

\setcounter{case}{0}

\section{Embeddings of Groups of order 12}\label{order12}

We now proceed to show that there is a unique group of minimal order in which all groups of order 12 can be embedded, namely $ G = S_3 \times S_4$. We first show that all groups of order 12 can indeed be embedded in $ S_3 \times S_4$.

\begin{lemma}\label{lem4}
	Let $G = S_{3} \times S_{4}$. Then all groups of order 12 can be embedded in $G$.
\end{lemma}

\begin{proof}
We let $ S_{3} = \langle \ a, \ b \mid a^{3} = b^{2} = 1, \  bab = a^{-1} \ \rangle $ and express $ S_{4} $ in terms of the usual cycle notation. Up to isomorphism the groups of order 12 are: $  C_{12}, \ C_{2} \times C_{2}\times C_{3}, \ D_{6},  \ Q_{3} \mbox{\ and\ } A_{4}. $ The following can then be confirmed directly
\begin{eqnarray*}
C_{12} & \cong & \langle \ a(1234) \ \rangle ;\\
C_{2} \times C_{2}\times C_{3} & \cong & \langle \ a, \ (12)(34), \ (13)(24) \ \rangle ;\\
 D_{6} & \cong & \langle \ b, \ a(12) \ \rangle ; \\
Q_{3} & \cong & \langle \ a(13)(24), \ b(1234) \ \rangle . 
\end{eqnarray*}
Since $ A_4 \leqslant S_4 $, this shows that all groups of order 12 can be embedded in $G$.
\end{proof}

We show next that the order of $ S_3 \times S_4 $ is minimal with respect to the embedding of all groups of order 12.

\begin{lemma}
	Let $G$ be a group of minimal order such that every group of order 12 can be 
	embedded in G. Then $|G| = 144$.
\end{lemma}

\begin{proof} 
From Lemma \ref{lem4} we see that $|G| \le | S_3 \times S_4 | = 144$. On the other hand we note that $G$ contains subgroups isomorphic to $C_{12} \cong C_{3} \times C_{4} $ and $C_{2} \times C_{2} \times C_{3} $. In particular $G$ has subgroups isomorphic to $C_{4}$ and $C_{2}\times C_{2}$. Letting $S$ be a Sylow 2-subgroup of $G$  it follows from Sylow's Theorems that $S$ has subgroups isomorphic to both $C_{4}$ and $C_{2}\times C_{2}$. Thus $S$ cannot be cyclic and must have order at least 8.

Now let $P$ be a Sylow 3-subgroup of $G$. Then $P$ is non-trivial, so $|P| \ge 3$. Suppose that $|P|= 3$. Then all Sylow 3-subgroups of G are isomorphic to $C_{3}$. Since all Sylow 3-subgroups are conjugate in $G$, and since $G$ has subgroups isomorphic to $ C_{3} \times C_{4} $ and $ C_{2} \times C_{2} \times C_{3} $ we see that $C_{G}(P)$ has subgroups isomorphic to $C_{4}$ and $ C_{2} \times C_{2} $.

Letting $S_{1}$ be a Sylow 2-subgroup of $C_{G}(P)$, and arguing as above, we see that $|S_{1}| \ge 8$. In addition $G$ has a subgroup isomorphic to $D_{6}$ so there must be a 2-element in $G$ that normalizes, but does not centralize, $P$. It follows that 2 is a divisor of $ |N_{G}(P) : C_{G}(P)|$. Since $P \leqslant C_{G}(P) $ we have
\[
 	|N_{G}(P)| \ge 2|C_{G}(P)| \ge  2.8.3 = 48.   
\]
Furthermore $G$ has a subgroup, $U$, which is isomorphic to $A_{4}.$ Since all Sylow 3-subgroups are conjugate in $G$ we may assume that  $ P  \leqslant U $.   However the Sylow 3-subgroups of  $A_{4}$  are self-normalizing in  $A_{4},$  so we have   $ U \cap N_{G}(P)  =  P $ and it follows that
\[ 
	|G| \ge  |U N_{G}(P)| 
		=    \frac{|U||N_{G}(P)|}{|U \cap N_{G}(P)|}  
		\ge  \frac{12.48}{3}
		=    192, 
\]
which contradicts the minimality of $|G|$.

Thus we see that $|P| \ge 9$.  If $|P| \ge 27$ then $|G| \ge 8.27 = 216$, which is ruled out, so $|P| = 9$. Hence $P$ is isomorphic either to $ C_{9} $ or $ C_{3} \times C_{3} $. If $ P \cong C_{9} $ then $P$ has a unique subgroup of order 3 and, by Sylow's Theorems, all subgroups of order 3 in $G$ are conjugate in $G$. Letting $ P_{1} $ be the unique subgroup of order 3 in $P$, we can use the argument outlined above to show that $ |N_{G}(P_{1})| \ge 48 $ and $ |G| \ge 192$. Therefore  we have $ P \cong C_{3} \times C_{3}$. Now, since $ |P| = 9 $ and $ |S| \ge  8$, we see that  $G$ is a multiple of 72. But since $ |G| \le 144 $ this means that $ |G| = 144 $ or $ |G| = 72 $.

Suppose that $ |G| = 72$. Then, since $G$ has a subgroup isomorphic to $ C_{2} \times C_{2} \times C_{3} $, there must be an element $x$ of order 3 in $G$ such that $ \langle x \rangle$ is centralized by a subgroup isomorphic to $ C_{2}\times C_{2} $. In addition the Sylow 3-subgroups of $G$ are abelian so $ \langle x \rangle$ is centralized by a subgroup of order 9. Hence $ |C_{G}(x)| \ge 2^{2}.3^{2} = 36 $. Therefore we have either $ C_{G}(x) = G $, whence $ x  \in Z(G) $, or $ |G : C_{G}(x)| = 2 $ and $ C_{G}(x)  \unlhd G$. In the latter case we have
\[
G/C_{G}(x) \cong C_{2}.
\]
Therefore every element of order 3 in $G$ is contained in $ C_{G}(x) $. For $U$ as above with $ U \cong A_4 $, we see that, since $U$ is generated by elements of order 3, $ U \leqslant C_{G}(x) $. Now $ Z(U) =1$, so  $ x \notin U $. By comparison of orders we then have
\[ 
C_{G}(x) = \langle x \rangle \times U.
\]
Since $G$ has a subgroup isomorphic to $ C_{12} \cong C_{3} \times C_{4} $ we see that there is an element of order 3, say $ x_{1} \in G $, that is centralized by an element of order 4, say $ y \in G $. Since  $ G/C_{G}(x) \cong C_{2} $ we see that  $ \langle x_{1}, y^{2} \rangle  \leqslant  C_{G}(x)  \leqslant  \langle x \rangle \times U $. Thus $ x_{1} $ is an element of order 3 in  $ \langle x \rangle \times U $ ($ \cong C_3 \times A_4 $) that is centralized by the element $ y^{2} $, which has order 2 and is contained in $U$. It follows that $ x_{1} \in \langle x \rangle $  and so $ \langle x_{1} \rangle = \langle x \rangle $. But then $x$ is also centralized by $ y \notin C_{G}(x) $, which is a contradiction. Thus we conclude that $ x \in Z(G) $.

Now $G$ has a subgroup $Q$, say, which is isomorphic to $ Q_{3} $. Since the Sylow 3-subgroup of $Q$ is not central in $Q$, we see that $ Q \cap \langle x\rangle = 1$. Hence
\[
Q\langle x \rangle/\langle x \rangle \cong Q \cong Q_{3},
\]
and likewise we have
\[
U\langle x \rangle/\langle x \rangle \cong U \cong A_{4}.
\]
Hence
\begin{align*} 
	|G/\langle x\rangle|   &\ge   |(Q\langle x\rangle/\langle x\rangle)(U\langle x\rangle/\langle x\rangle)| \\[4pt]
	 &=  \frac{|(Q\langle x \rangle/\langle x \rangle)||(U\langle x \rangle/\langle x \rangle)|}{|(Q\langle x \rangle/\langle x \rangle) \cap (U\langle x \rangle/\langle x \rangle)|}\\[4pt]
	& =  \frac{|Q_{3}||A_{4}|}{|(Q\langle x \rangle/\langle x \rangle) \cap (U\langle x \rangle/\langle x \rangle)|}.
\end{align*}
Thus
\[
\frac{72}{3} = 24 \ge \frac{144}{|(Q\langle x \rangle/\langle x \rangle) \cap (U\langle x \rangle/\langle x \rangle)|},
\]
so it follows that $ |(Q\langle x \rangle/\langle x \rangle) \cap (U\langle x \rangle/\langle x \rangle)| \ge 6 $.
Since $ Q\langle x \rangle/\langle x \rangle $ and $ U\langle x \rangle/\langle x \rangle $ are non-isomorphic groups, both of order 12, we conclude: 
\[
|(Q\langle x \rangle/\langle x \rangle) \cap (U\langle x \rangle/\langle x \rangle)| = 6.
\]
In particular $ U\langle x \rangle/\langle x \rangle $ has a subgroup of order 6. But $ U\langle x \rangle/\langle x \rangle $ is isomorphic to $ A_{4} $ which does not have any subgroups of order 6 and a contradiction ensues. Hence we conclude that $ |G| = 144$.
\end{proof}

The final result in the present section shows that $S_3 \times S_4$ is the unique group of minimal order in which all groups of order 12 can be embedded.

\begin{theorem}\label{12S3SX4}
Let $G$ be a group of order 144 in which all groups of order 12 can be embedded. Then $ G \cong S_3 \times S_4 $.
\end{theorem}

\begin{proof}

Let $P$ be a Sylow 3-subgroup of $G$ and let $U_1$, $U_2$ and $U_3$ be subgroups of $G$ such that $ U_1 \cong C_{12} \cong C_3  \times C_4 $, $ U_2 \cong C_2 \times C_2 \times C_3 $ and $U_3 \cong A_4$. By Sylow's Theorems we may assume that $ P \cap U_i \cong C_3 $, for $ i = 1,2,3 $. Now $|P| = 9$, so we have $ P \cong C_3 \times C_3 $ or $ P \cong C_9 $. 

If $ P \cong C_9 $ then $P$ has a unique subgroup of order 3, $P_1$ say, so $U_i \cap P = P_1$ ($ i = 1,2,3)$. Since $U_1$, $U_2$ and $P$ are abelian  we have $ \langle \ U_1, \ U_2, \ P \ \rangle \leqslant C_G(P_1) $. Then the Sylow 2-subgroups of $ C_G(P_1)$ have subgroups isomorphic to both $C_4$ and $C_2 \times C_2$, and so have order at least 8. Hence $ |C_G(P_1)| \geq 8 \times 9 = 72 $ so $C_G(P_1)$ has index at most 2 in $G$, and it follows that $C_G(P_1)$ is normal in $G$. Now $P_1 \leqslant Z(C_G(P_1)) $ so $ P_1 \leqslant O_3(G) $. Since the Sylow 3-subgroups of $G$ are cyclic, it follows that $P_1$ is the unique subgroup of order 3 in $G$. But $U_3$ is generated by elements of order 3 and the contradiction $ A_4 \cong U_3 \leqslant P_1 \cong C_3 $ follows. Hence $ P \cong C_3 \times C_3 $.

Now $G$ is soluble (e.g. by Burnside's Theorem), so $F(G)$ is non-trivial. Since $G$ is a $\{2, 3\}$-group we have $ F(G) = O_2(G) \times O_3(G) $. If $ | O_2(G) | = 16 $ then $G$ has a normal Sylow 2-subgroup. But then $D_6$ and $Q_3$ cannot be embedded in $G$. Thus we have $ |O_2(G)| \leq 8 $. Similarly, if $ |O_3(G)| = 9 $ then $A_4$ cannot be embedded in $G$, so $ |O_3(G)| \leq 3 $. As no 2-group of order 8 or less has an automorphism group in which $C_3 \times C_3 $ can be embedded, we see that $ C_P(O_2(G)) \neq 1 $. Since $P$ is abelian and $O_3(G) \leqslant P $, we have $ C_P( O_2 (G) ) \leqslant C_P ( O_3(G) \times O_2(G) ) = C_P( F(G) ) \leqslant F(G) $. Thus $ 1 \neq C_P( O_2(G) ) \leqslant O_3(G) $ and it follows that $ O_3(G) \cong C_3 $.

If $ 3 \nmid |\bmaut(O_2(G) ) | $ then, since $ F(G) = O_3(G) \times O_2(G) $ and $ O_3(G) \cong C_3 $, we have $ 3 \nmid | \bmaut(F(G) )| $ and the contradiction $ P \leqslant C_G( F(G) ) \leqslant F(G) $ ensues. Hence $ 3 \mid | \bmaut( O_2(G) ) | $ so, in particular, $O_2(G) $ is non-cyclic. If $O_2(G)$ is non-abelian then $|O_2(G)| = 8$ and $ O_2(G) \cong D_4 $ or $ O_2(G) \cong Q_2 $. If $ O_2(G) \cong D_4 $ then, since $ \bmaut(D_4) \cong D_4 $, the contradiction $ P \leqslant F(G) $ arises as above. If $O_2(G) \cong Q_2 $ then $ |F(G)| = 24 $ and since $ U_3 \cong A_4 $ we have, by comparison of orders, $ 1 \neq U_3 \cap F(G) $. By minimal normality we then have the contradiction:
\[ C_2 \times C_2 \cong U_3^{'} \leqslant O_2(G) \cong Q_2 .\]
Thus $O_2(G)$ is abelian and non-cyclic, and so is isomorphic either to $ C_2 \times C_4 $, \ $C_2 \times C_2 \times C_2 $ or $ C_2 \times C_2 $.

If $ O_2(G) \cong C_2 \times C_4 $ then $ 3 \nmid | \bmaut ( O_2 (G) ) | $, which is ruled out. If $ O_2(G) \cong C_2 \times C_2 \times C_2 $ we again see by minimal normality that $ U_3^{'} \leqslant O_2(G) $. We let $ \langle x \rangle $ be a Sylow 3-subgroup of $ U_3$ and apply Maschke's Theorem to see that there exists $ z \in O_2(G) $ such that $ \langle z \rangle $ is normalised by $x$ and such that $ O_2(G) = \langle z \rangle \times U_3^{'} $. Since $x$ also centralises $O_3(G)$ we have
\[ F(G) U_3 = O_3(G) \times \langle z \rangle \times U_3  \ ( \cong C_3 \times C_2 \times A_4 ) .\]
Now $ | G : F(G) U_3 | = 2 $, so $ F(G) U_3 \unlhd G $. Since $ \langle z \rangle $ is characteristic in $ F(G) U_3 $ we have $ \langle z \rangle \unlhd G $ and, in particular, $ z \in Z(G) $. In addition, since $ U_3^{'} = ( F(G) U_3 )^{'} $, we have $ U_3^{'} \unlhd G $.

Since $ D_6 $ can be embedded in $G$ there must be an element $ y \in G $ such that $ o(y) = 2 $ and $ y \notin O_2(G) $. By Maschke's Theorem $ O_3(G) O_2(G) / O_2(G) $ has a complement, $ W / O_2(G) $, in $ F(G) U_3 / O_2(G) $ that is normalised by $ \langle y \rangle O_2(G) / O_2 (G) $. Then $ W \unlhd G $ and $ W \cap O_3(G) \leqslant O_2(G) \cap O_3(G) = 1 $, so $ F(G) U_3 = W \times O_3(G) $ and it follows that $ W \cong \langle z \rangle \times U_3 $. Hence we may assume that $ \langle z \rangle \times U_3 $ is normal in $G$. But $ U_3 = O^{3^{'}}( \langle z \rangle \times U_3 ) $ is characteristic in $ \langle z \rangle \times U_3 $, so, in particular, we have $ U_3 \unlhd G $.

If  $y$ centralises $ U_3^{' } $ then the Sylow 2-subgroups of $G$ are elementary abelian, so $C_{12}$ cannot be embedded in $G$. Thus $y$ does not centralise $U_3^{'}$ and it follows that $ U_3 \langle y \rangle $ is isomorphic to $S_4$. Hence
\[ \langle z \rangle U_3 \langle y \rangle = \langle z \rangle \times U_3 \langle y \rangle \cong C_2 \times S_4 . \]
If $y$ centralises $O_3(G)$ then
\[ G = O_3(G) \times \langle z \rangle \times U_3 \langle y \rangle \cong C_3 \times C_2 \times S_4 .\]
Since $ G / O_3(G) U_3 $ has exponent 2 and $ O_3(G) U_3 / U_3^{'} \cong C_3 \times C_3 $ we see that if $ h \in G $ is such that $ o(h) = 4 $ then $ h^2 \in U_3^{'}$. Now $ Q_3$ can be embedded in $G$ so there exists $ g \in G $ such that $ o(g) = 4 $ and such that $g$ normalises, but does not centralise, a subgroup of order 3 in $G$. But $ g \in O^{2^{'}}(G) = \langle z \rangle \times U_3 \langle y \rangle \unlhd G$ so we may assume, without loss of generality, that $g$ normalises $ \langle x \rangle $. Then $g^2$ centralises $x$. But $ g^2$ is a non-trivial element of $U_3^{'}$ and a contradiction arises. Hence we may assume that $ y $ does not centralise $O_3(G)$. 

It follows that $ O_3(G) \langle y \rangle \cong S_3$, so $ O_3(G) \langle y \rangle \cap O_2(G) = 1 $ and $ O_3(G) \langle y \rangle O_2(G) / O_2(G) \cong S_3 $. Since $ U_3 \langle y \rangle \cong S_4 $, we now see that $ y O_2(G) $ inverts every non-trivial element of $ O_3(G) O_2(G) / O_2(G) \times U_3 O_2(G) / O_2(G) $ ($ \cong C_3 \times C_3 $) so, in particular, $C_6$ cannot be embedded in $ G / O_2(G) $. But for $ U_1 \cong C_{12} $ (as above) we see that $ U_1 \cap O_2(G) \cong C_2$ (since $ O_2(G) \cong C_2 \times C_2 \times C_2 $) and the contradiction $ U_1 O_2(G) / O_2(G) \cong C_6 $ arises. Thus $ O_2(G) $ cannot be isomorphic to $ C_2 \times C_2 \times C_2 $ and we conclude that $ O_2(G) \cong C_2 \times C_2 $.

Now $ U_3 \cong A_4 $ so $ U_3^{'}$ centralises $O_2(G)$. If $ U_3^{'} \cap O_2(G) = 1 $ then $ O_2(G) U_3^{'} = O_2(G) \times U_3^{'} \cong C_2 \times C_2 \times C_2 \times C_2 $ and the Sylow 2-subgroups of $G$ are elementary abelian. But then $C_{12}$ cannot be embedded in $G$, which is a contradiction. Hence $ U_3^{'} \cap O_2(G) \neq 1 $ so, by minimal normality, we have $ U_3^{'} = O_2(G) $. Since $F(G) = O_3(G) \times O_2(G) \cong C_3 \times C_2 \times C_2 $ we have 
\[ \bmaut( F(G) ) \cong \bmaut (C_3 ) \times \bmaut( C_2 \times C_2 ) \cong C_2 \times S_3 .\]
Now $F(G)$ is abelian so $ C_G(F(G)) = F(G) $. Then, by comparison of orders, we have
\[ G / F(G) \cong C_2 \times S_3 . \]
Since $ G / F(G) $ thus has a quotient group isomorphic to $ C_2 \times C_2 $ we see that $ | O^{3^{'}} (G) | \leq 36 $. On the other hand $ O_3(G) U_3 \leqslant O^{3^{'}} (G) $ and $ | O_3(G) U_3 | = 36 $, so it follows that $ O_3(G) U_3 = O^{3^{'}} (G) \unlhd G $. Furthermore, since $ O_3(G) \cong C_3 $ and $ U_3 \cong A_4 $, we see that $U_3$ centralises $O_3(G)$ so $ O^{3^{'}} (G) = O_3(G) \times U_3$ ($ \cong C_3 \times A_4 $).

Now $ O^{3^{'}} (G) / O_2(G) \cong C_3 \times C_3 $ and $G / O^{3^{'}} (G) $ is a 2-group so we apply Maschke's Theorem to see that $ O^{3^{'}} (G) / O_2(G) = O_3(G) O_2(G) / O_2(G) \times K / O_2(G) $, where $ K / O_2(G) \cong C_3 $ and $ K \unlhd G$. We have $ O_3(G) \cap K \leqslant O_3(G) \cap O_2(G) = 1 $ and $ O^{3^{'}} (G) = O_3(G) K $. Hence $ O^{3^{'}} (G) = O_3(G) \times K $ and $ K \cong O^{3^{'}} (G) / O_3(G) \cong A_4 $. Thus we may assume, without loss of generality, that $ U_3 = K \unlhd G$. We let $ W = C_G( O_3(G) ) $. Since $ G / F(G) \cong C_2 \times S_3 $ we have $ O_3(G) \not\leqslant Z(G) $, and hence $ G / W \cong C_2 $ and $ |W| = 72 $. Now $ U_3 \leqslant W $ and $ U_3 \cong A_4 $. Thus $ C_W(U_3) \cap U_3 = 1 $ and $ W / C_W( U_3 ) $ is isomorphic either to $A_4$ or $S_4$. Since $|W| = 72$ we must then have $ |C_W( U_3 )| = 6 $ or $ |C_W( U_3 )| = 3 $.

If $|C_W( U_3 )| = 6 $ then, since $ O_3(G) $ centralises $U_3$, we have $ C_W(U_3) \cong C_6 $ and we may let $ C_W( U_3 ) = O_3(G) \times \langle x \rangle $, where $ o(x) = 2 $. Then
\[ W = C_W( U_3 ) \times U_3 = O_3(G) \times \langle x \rangle \times U_3 , \]
so $ \langle x \rangle \leqslant O_2(W) \leqslant O_2(G) = U_3^{'} $, which is a contradiction. Hence $ |C_W( U_3) | = 3 $ so, in particular, $ C_W( U_3 ) = O_3(G) $.

Since $ | W : O_3(G) \times U_3 | = 2 $ there exists a 2-element $ y \in W \backslash ( O_3(G) \times U_3 ) $ such that $ W = ( O_3 (G) \times U_3 ) \langle y \rangle $. Now $y$ centralises $O_3( G )$ and $ y^2 \in O_3(G) \times U_3 $. But $y^2$ is a 2-element, so $ y^2 \in U_3^{'} $. Thus $ U_3 \langle y \rangle $ is a group of order 24 and 
\[ W =  O_3 (G) \times U_3 \langle y \rangle. \]
We further have $ U_3 \langle y \rangle \cong W / O_3(G) = W / C_W( U_3 ) \cong S_4 $. In addition we see that $ U_3 \langle y \rangle = O^{2^{'}} (W) \unlhd G $. We let $ C = C_G( U_3 \langle y \rangle ) $. Then $ C \cap U_3 \langle y \rangle = 1 $ and, since $S_4$ is complete, we have $ G = C \times U_3 \langle y \rangle $. If follows that $ |C| = 6 $. If $ C \cong C_6 $ then $ C_2 \cong O_2(C) \leqslant O_2(G) = U_3^{'} $, which is again a contradiction. Thus $ C \cong S_3 $ and so $ G  = C \times U_3 \langle y \rangle  \cong S_3 \times S_4 $, as desired.
\end{proof}

Bearing Lagrange's Theorem in mind, it is of interest to consider whether the minimal order of a group in which a given collection of groups can be embedded will always be a divisor of the order of any other group in which the same collection of groups can be embedded. The following example shows that this will not always be the case.

\begin{example}
\end{example}

We let
\[ G = \langle x \rangle \times \langle y \rangle \times S_4 , \]
where $ \langle x \rangle \cong C_2 $ and $ \langle y \rangle \cong C_4$. The following isomorphisms can be confirmed: 
\begin{eqnarray*}
C_{12} & \cong & \langle \ y, \ (123) \ \rangle  ;\\
C_2 \times C_2 \times C_3  & \cong & \langle \ x, \ y^2, \ (123) \ \rangle;\\
 D_6  & \cong & \langle x\rangle \times S_3 ;\\
 Q_3 & \cong & \langle \ y(12), \ (123) \ \rangle .
\end{eqnarray*}
As $A_4$ is a subgroup of $S_4$, we see that all groups of order 12 can be embedded in $G$. We have $ |G| = 192 $ and hence 144, which is the minimal order of a group in which all groups of order 12 can be embedded,  is not a divisor of $|G|$.

\section{Minimal embeddings of groups of order $n$ for $ n \leq 15 $}\label{ordernleq15}

Having dealt with the cases $n = 8$ and $n = 12$, we present the groups of minimal order in which all groups of order $n$ can be embedded for $ n \leq 15 $. For reference, we recall that the groups of order 15 or less are as follows:

\begin{table}[H]
\begin{center}
\begin{tabular}{cccc}
\hline
$ n $&Groups of order $ n $& $ n $&Groups of order $ n $\\
\hline
\rule{0pt}{3ex}
$1$&$ C_{1}$&$ 9 $&$ C_{9}, C_{3} \times C_{3} $\\
$ 2 $&$ C_{2} $&$ 10 $&$ C_{10}, D_{5} $\\
$ 3 $&$ C_{3} $&$ 11 $&$ C_{11} $ \\
$ 4 $&$ C_{4}, C_{2} \times C_{2} $&$ 12 $&$ C_{12}, C_{2} \times C_{2} \times C_{3}, D_{6}, Q_{3}, A_{4} $\\
$ 5 $&$ C_{5} $&$ 13 $&$ C_{13} $\\
$ 6 $&$ C_{6}, D_{3} $&$ 14 $&$ C_{14}, D_{7} $\\
$ 7 $&$ C_{7} $&$ 15 $&$ C_{15} $\\
$ 8 $&$ C_{8},   C_{2} \times C_{4}, C_{2} \times C_{2}\times C_{2} , D_{4}, Q_{2} $\\
\hline
\end{tabular}
\caption{Groups of order 15 or less}\label{|G|leq15}
\end{center}
\end{table}

In order to provide a description of the minimal groups for $ n = 9 $ we let $H_3$ denote the non-abelian group of order 27 and exponent 9, that is $ H_{3} = \langle \ a, \ b \mid a^{9} = 1 = b^{3}, \  b^{-1}ab = a^{4} \ \rangle $. Taking the results from Sections \ref{order8} and \ref{order12}, and noting that the remaining cases are quite straightforward, our next proposition gives the groups of least order in which all groups of order $n$ can be embedded, for $ 1 \leq n \leq 15$.

\begin{theorem}\label{embednleq15}

Let  $ 1 \leq n \leq 15$ and let $G$ be a group of minimal order in which all groups or order $n$ can be embedded. Then the isomorphism class of $G$ is as in the following table:

\def\TTT{ \begin{tabular}{@{}c@{}}Group(s) of\\[-3pt]minimal order\\ \end{tabular}}

\begin{table}[H]
\begin{center}
\begin{tabular}{cccccccc}
\hline
$n$&\TTT&& $n$&\TTT&& $n$&\TTT\\
\hline
\rule{0pt}{3ex}
$1$ & $C_{1}$ && $6$ & $D_{6}$ && $11$ & $C_{11}$ \\ 
$2$ & $C_{2}$ && $7$ & $C_{7}$ && $12$ & $ S_3 \times S_{4}$\\
$3$ & $C_{3}$ && $8$ & $ C_{2} \times H_1, \ H_2  $ && $13$ & $C_{13}$\\
$4$ & $C_{2} \times C_{4}, \ D_{4}$ && $9$ & $C_{3} \times C_{9}, \ H_{3}  $ && $14$ & $D_{14} $\\
$5$ & $C_{5}$ &&  $10$ & $D_{10}$ && $15$ & $C_{15}$\\
\hline
\end{tabular}
\end{center}
\caption{Group(s) of least order containing all groups of order $n$, for $ 1 \le n \le 15 $ }\label{MinGroups1}
\end{table}

\end{theorem}

\section{Embeddings of groups of order $p^3$, $p$ odd}\label{p-odd}

Theorem \ref{embednleq15} tells us in particular that the lower bound given by Lemma \ref{pbound} is attained for $n = $ 4, 8 and 9.                           The following result shows that this bound is not attainable for $n = p^k $, where $p$ is an odd prime and $ k \geq 3$.

\begin{lemma}\label{p^knotembed}

Let $p$ be an odd prime and let $k \geq 3 $. If $G$ is a group of order $ p^{2k-1}$ then not all groups of order $p^k$ can be embedded in $G$.

\end{lemma}

\begin{proof}

We suppose that $ C_{p^k} $ and the elementary abelian group of rank $k$ can be embedded in $G$ and let $A $ and $B$ be subgroups of $G$ such that $ A = \langle x \rangle \cong C_{p^k} $ and $B$ is elementary abelian of rank $k$.                   Since $A$ is cyclic and $ |G| = p^{ 2k - 1 } $ we must have $ | A \cap B | = p $ and $ G = AB $.                    We can then apply \cite{McCann} Lemma 2.5  to see that $ B =  \Omega_1(A) B \unlhd G $. In particular we have $ G / B \cong A / ( A \cap B ) \cong C_{p^{k-1}} $. If $ k \geq 4 $ and $U$ is a subgroup of $G$ such that $ U \cong C_{p^{2}} \times C_{p^{k-2}} $, then $ U \cap B $ is elementary abelian, so $ UB / B \cong U / ( U \cap B ) $ has a subgroup isomorphic to $ C_p \times C_p $. But $ G / B $ is cyclic and a contradiction ensues. Hence not all groups of order $ p^k $ can be embedded in $G$.

If $ k = 3 $ then $ B \cong C_p \times C_p \times C_p $ and the Sylow p-subgroups of $ \bmaut(B) $ are isomorphic to those of $ GL( 3, p ) $. Since $p$ is odd, we see that 
$ P = \{ \ \left( \begin{array}{lll}
1 & 0 & 0\\
a & 1 & 0\\
b & c & 0\\
\end{array} \right) \ \mid \ a, \ b, \ c \in GF( p) \ \} $
is a Sylow $p$-subgroup of $ GL( 3, p )$. $P$ is non-abelian, of order $p^3$, and has exponent $p$. Hence $ G / C_G(B) $ has exponent $p$. Since $B$ is abelian and $ G / B \cong C_{p^2} $ we then have $ | C_G(B) : B | \ge p $, so $ \langle x^p \rangle \leqslant C_G(B) $. Now $ \langle x^p \rangle B / B $ is the unique subgroup of order $p$ in $ G / B  $, so all elements of order $p$ in $G$ are contained in $ \langle x^p \rangle B $. But $ \langle x^p \rangle B $ is abelian, since $x^p$ centralises B. Hence the non-abelian group of order $p^3$ and exponent $p$ cannot be embedded in $G$.

\end{proof}

Lemma \ref{p^knotembed} shows that if $p$ is a odd prime and $G$ is a group of minimal order in which all groups of order $p^3$ can be embedded, then $ |G| \geq p^6 $. Our next example shows that this latter bound can be attained.

\begin{example}
\end{example}

Let $p$ be an odd prime and let $ A = \langle x \rangle $ be cyclic of order $p^3$. Let $B$ be the non-abelian group of order $p^3$ and exponent $p$ given by
\[ B = \langle \ a, \ b, \ z_1 \mid a^p = b^p = z_1^p = 1 , \ [ a, z_1 ] = [ b, z_1 ] =1, \ b^{-1}ab = az_1 \ \rangle \]
We form the direct product $ W = A \times B $ and let $ z_2 = x^{p^2} $.
We then let
\[ W_1 = W / \langle \ z_1^{-1}z_2 \ \rangle . \]
$ W_1 $ is thus the central product of $ A  \langle \ z_1^{-1}z_2 \ \rangle / \langle \ z_1^{-1}z_2 \ \rangle $ and $ B \langle \ z_1^{-1}z_2 \ \rangle / \langle \ z_1^{-1}z_2 \ \rangle $, in which the elements $z_1$ and $ z_2$ have been identified with each other. We let $ z = z_1 \langle \ z_1^{-1}z_2 \ \rangle$  ($= z_2 \langle \ z_1^{-1}z_2 \ \rangle $). For notational convenience, we write $ \langle \  z_1^{-1}z_2 \ \rangle = 1 $. Then $W$ is the central product of $A$ and $B$, with $ A \cap B = \langle z \rangle $. In particular we have $ |W| = \frac{|A| |B| }{ | A \cap B | }= \frac{p^3p^3}{p} = p^5 $. We can express $W$ in terms of generators and relations as follows:
\[ W = \langle \ a, \ b, \ x, \ z \mid a^p = b^p = z^p = x^{p^3} = 1, \ [ a, z ] = [ b, z ] = 1, \ b^{-1}ab = az, \ [ a, x ] = [ b, x ] = 1 , \  x^{p^2} = z \ \rangle .\]

By construction $  A \cong C_{p^3} $, and $ B $ is isomorphic to the non-abelian group of order $ p^3$ and exponent $p$. We can confirm that $ \langle \ a, \ x^p \ \rangle \cong C_p \times C_{p^2}  $, and that $ \langle \ ax^p, \ b \ \rangle $ is isomorphic to the non-abelian group of order $p^3$ and exponent $p^2$. Thus, apart from $ C_p \times C_p \times C_p $, all groups of order $p^3$ can be embedded in $W$. We now let
\[ G = W \times C_p .\]
Then, since $ \langle \ a, \ x^{p^2} \ \rangle \cong C_p \times C_p $ is a subgroup of $W$, we see that $G$ is a group of order $p^6$ in which all groups of order $p^3$ can be embedded.

\section{Embeddings of groups of order 15 or less: the soluble case}\label{soluble}

We now turn to the question of determining the  least order of a group, $G$,  in which all groups of order $n$ or less can be embedded for $ n \leq 15$. We first consider the case where $G$ is soluble. If $ n \geq 12 $ then all groups of orders 8, 9 and 12 can be embedded in the Hall $\{ 2, 3 \}$-subgroups of $G$. Lemma \ref{nbound} then shows that the order of a Hall \{2,3\}-subgroup of $G$ is a multiple of $ 2^5.3^3 = 864$. Here we show that the bound given by Lemma \ref{nbound} is not attainable by soluble groups for $ 12 \leq n \leq 15 $. We first derive some properties of groups of order $ 2^5.3^k$. In this section all groups considered will be soluble.

\begin{lemma}\label{2^53^kO3G=1}

Let $G$ be a group of order $ 2^5.3^k $ such that $ O_3(G) =1 $. Then $ k \leq 2 $ and, where $ k = 2 $, the Sylow 2-subgroups of $G$ have exponent at most 4.

\end{lemma}

\begin{proof}

Since $ O_3(G) =1 $ we see that $ F(G) $ is a non-trivial 2-group. Letting $P$ be a Sylow 3-subgroup of $G$, we have
\[ C_P( F(G) ) \leqslant C_G( F(G) ) \leqslant F(G). \]
Since $C_P( F(G) )$ is a 3-group we have $ C_P( F(G) ) = 1 $, so $P$ acts faithfully on $F(G)$ and hence on the elementary abelian 2-group $ F(G) / \Phi ( F(G) ) $.

Now $ F(G) / \Phi ( F(G) ) $ has rank at most 5, and since $ 3^2 \top | GL( 5, 2 ) | $, we see that $ |P| \leq 9 $. If $ |P| = 9 $  then $ | F(G) / \Phi ( F(G) ) | \geq 16 $. Since $ F(G) / \Phi ( F(G) ) $ is elementary abelian, the Sylow 2-subgroups of $G$ can have exponent at most 4 in this case. 

\end{proof}

\begin{corollary}\label{2^53^nC8}

Let $G$ be a group of order $ 2^5.3^k $. If $C_8$ can be embedded in $G$ then $ | O_3(G) | \geq 3 ^{k-1} $.

\end{corollary}

\begin{proof}

Let  $ W = O_3(G) $. Then $ O_3( G/W ) = 1_{G/W} $. Since $C_8$ can be embedded in $G$ we see that the exponent of the Sylow 2-subgroups of $G$, and hence of those of $G/W$, is at least 8. By Lemma \ref{2^53^kO3G=1} we see that $ |G/W| $ is at most $ 2^5.3 $, so it follows that $ | W | = | O_3(G) | \geq 3^{k-1} $.

\end{proof}

Our main result in this section is Lemma \ref{2^53A4}, which deals with embeddings in groups of order 96.

\begin{lemma}\label{2^53A4}

Let $G$ be a group of order  96 (= $2^5.3$) in which $A_4$ can be embedded. Then not all groups of order 8 can be embedded in $G$.

\end{lemma}

\begin{proof}

We let $ U \leqslant G$ be such that $ U \cong A_4 $. We suppose first that $ U = O^{3^{'}} (G) $. Then $ U \unlhd G $ and, letting $ C = C_G(U) $, we see that $ C \unlhd G $ and $ C \cap U = 1$. Since $ \bmaut(A_4) \cong S_4 $ and $ |G| = 2^5.3 $ we further see that $ |C| = 4 $ or $ |C| = 8 $.

If $ |C| = 4 $ then $ U^{'} \times C $ is an abelian subgroup of $G$ of order 16 and exponent at most 16, so we may apply Lemma \ref{Habex4} to see that not all groups of order 8 can be embedded in $G$. If $ |C| = 8 $ then $C$ has an abelian subgroup or order 4, $C_1$ say. Then $ U^{'} \times C_1 $ is an abelian subgroup of order 16 and exponent at most 4. We again apply Lemma \ref{Habex4} to see that not all groups of order 8 can be embedded in $G$. Thus we may assume that $U$ is a proper subgroup of $ O^{3^{'}} ( G ) $.

We let $ W = O_2 ( O^{3^{'}}( G ) ) $ and let $ \langle y \rangle $ be a Sylow 3-subgroup of $U$. We see that $ O^{3^{'}} ( G ) = W \langle y \rangle $ with $ [ \langle y \rangle, W ] = W $. By Maschke's Theorem, we have either $ W / \Phi ( W ) \cong C_2 \times C_2 $ or $ W / \Phi (W) \cong C_2 \times C_2 \times C_2 \times C_2 $. In the latter case the Sylow 2-subgroups of $G$ can have exponent at most 4, so $C_8$ cannot be embedded in $G$. Hence we may further assume that $ W / \Phi (W) \cong C_2 \times C_2 $.

By the minimal normality of $ U ^{'} $ in $U$ we see that if $ U^{'} \not\leqslant \Phi (W) $ then $ U^{'} \cap \Phi (W) = 1 $. Then by comparison of orders $ W / \Phi (W) = U^{'} \Phi (W) / \Phi (W) $, so $ W = U^{'} \Phi(W) = U^{'} $. But then $ U = \langle y \rangle U^{'} = \langle y \rangle W = O^{3^{'}} (G) $ and a contradiction ensues. Therefore $ U^{'} \leqslant \Phi (W) $ so, in particular, $ | \Phi (W) | \geq 4 $ and $ |W| \geq 16 $.

If $ | \Phi (W) | = 8 $ then, by comparison of orders, $W$ is a Sylow 2-subgroup of $G$. In addition since $ C_2 \times C_2 \cong U^{'} \leqslant \Phi (W) $ we see that either $ \Phi ( \Phi (W) ) = 1 $ or $ \Phi ( \Phi (W) ) \cong C_2 $. If $ \Phi ( \Phi (W) ) = 1 $ then $ \Phi (W) $ is elementary abelian and $W$ has exponent at most 4, so $C_8$ cannot be embedded in $G$. If $ \Phi ( \Phi (W) ) \cong C_2 $ then, by minimal normality,  $ \Phi ( \Phi (W) ) \cap U^{'} = 1 $ so $ \Phi (W) = U^{'} \times \Phi ( \Phi (W) ) \cong C_2 \times C_2 \times C_2 $ and the contradiction $ \Phi ( \Phi (W) ) = 1 $ ensues.

We may thus assume that $ | \Phi (W) | = 4 $, so $ \Phi (W) = U^{'} $ and $ |W| = 16 $. Since $ \bmaut ( U^{'} ) \cong S_3 $ and $ O^{3^{'}} ( G ) $ has no proper $ 3^{'} $-factor groups we see that $ W \leqslant C_G(U^{'} ) $, so $ U^{'} \leqslant Z(W) $. Furthermore, since $ W / U^{'} = W / \Phi (W) $ is abelian, it follows that $W$ is either abelian or else has class 2. Since $ W / \Phi( W) $ has rank 2 we may let $ W = \langle x, y \rangle $ for suitable elements $x$ and $y$. Then $ [ x, y ] \in \Phi (W)  = U^{'} \leqslant Z(W) $, so if $W$ is non-abelian we see that $ W^{'} = \langle [ x, y ] \rangle \cong C_2 $. By the minimal normality of $ U^{'} $ in $U$, we than have the contradiction $ U^{'} = W^{'} \cong C_2 $. It follows that $W$ is abelian. Since $ W / \Phi (W) \cong C_2 \times C_ 2 \cong \Phi (W) = U^{'} $ we see that $W$ has exponent at most 4. Hence $W$ is an abelian subgroup of order 16 and exponent at most 4 in $G$, so we may apply Lemma \ref{Habex4} to see that not all groups of order 8 can be embedded in $G$.

\end{proof}

\begin{corollary}\label{2^53^nA4}

Let $G$ be a group of order $ 2^5.3^k$ in which $A_4$ can be embedded. Then not all groups of order 8 can be embedded in $G$.

\end{corollary}

\begin{proof}

Let $U \leqslant G $ be such that $ U \cong A_4 $. By Corollary \ref{2^53^nC8} if $C_8$ can be embedded in $G$ then $ | O_3(G) | \geq 3^{k-1} $, so $ | G / O_3 (G) | $ is a divisor of $ 2^5.3 = 96 $. But $ U \cap O_3 (G) \leqslant O_3(U) =1 $, so $ 3 \mid | G / O_3 (G) | $. Hence $ | G / O_3 (G) | =  2 ^5.3 $ and $ A_4 \cong UO_3(G) / O_3(G) $. Since the Sylow 2-subgroups of $ G / O_3 (G) $ are isomorphic to those of $G$ we may now apply Lemma \ref{2^53A4} to see that not all groups of order 8 can be embedded in $G$.

\end{proof}

As a consequence of Lemma \ref{nbound} and Corollary \ref{2^53^nA4} we have the following result that establishes a lower bound for the order of a soluble group in which all  groups order $n$ or less can be embedded for $ 12 \leq n \leq 15 $. It will be shown in Section \ref{leqnleq15} that the bound given below can be attained in these cases.

\begin{corollary}\label{boundnleq15}

\begin{enumerate}[(a)]

\item 
Let $G$ be a finite soluble group in which all groups of order 12 or less can be embedded. Then $|G|$ is a multiple of $ 2^6.3^3.5.7.11 = 665  280 $.
\item
Let $G$ be a finite soluble group in which all groups of order $n$ or less can be embedded for $ n =$ 13, 14 or 15. Then $|G|$ is a multiple of $ 2^6.3^3.5.7.11.13 = 8648640 $.
\end{enumerate}

\end{corollary}

\section{Embeddings of groups of order 15 or less: the non-soluble case}\label{non-soluble}

Turning to embeddings in non-soluble groups, we prove a result analogous to Corollary \ref{boundnleq15} for the case where $G$ is non-soluble.  Due to the presence of non-abelian chief factors, the non-soluble case turns out to be considerably more complicated than the soluble case. We begin with some elementary results that show how normal quasisimple subgroups and central products may arise under certain circumstances.

\begin{lemma}\label{B'CA}

Let  $A$  and  $B$   be subgroups of the group $G$ such that
\begin{enumerate}[(i)]
\item
$ B  \leqslant  C_{G}( [ A, B] ) ;$
\item
$ [A,B]$  is abelian.
\end{enumerate}
Then $ B^{'}  \leqslant  C_{G}(A)$.
  
\end{lemma}

\begin{proof}

Let  $a$  be an element of $A$  and let $b_{1}$  and  $b_{2}$ be elements of $B$. Then
\begin{eqnarray*}
a^{b_{1}b_{2}} & = & (a^{b_{1}})^{b_{2}} \\
& = & (a[a, b_{1}] )^{b_{2}} \\
& = & a^{b_{2}}[a, b_{1}]\\
& = & a[a,b_{2}] [a, b_{1}] \\
& = & a[a, b_{1}] [a,b_{2}] \\
& = & a^{b_{1}}[a, b_{2}] \\
& = & (a [a,b_{2}])^{b_{1}} \\
& = &( a^{b_{2}})^{b_{1}} \\
& = & a^{b_{2}b_{1}}.
\end{eqnarray*}
Thus $ a^{b_{1}b_{2}} = a^{b_{2}b_{1}}$, so $ a^{b_{1}b_{2}b_{2}^{-1}b_{1}^{-1}} = a^{b_{2}b_{1}b_{2}^{-1}b_{1}^{-1}} $ and we have 
$ a  =  a ^{ [b_{2}^{-1}, b_{1}^{-1}]} $.
Since this holds for all $ b_{1}, b_{2}  \in B$, we conclude that $ B^{'}  \leqslant C_{G}(A) $.

\end{proof}

\begin{corollary}\label{CGN}

Let  $N  \unlhd G$  be such that $N  =  N^{'}$. Then  $ C_{G}(N/Z(N))  =  C_{G}(N) $.

\end{corollary}

\begin{proof}

Let $ C/Z(N) = C_{G/Z(N)}(N/Z(N))$. Then $ C = C_{G}(N/Z(N))$ and $ [ C, N] \leqslant Z(N) $. Hence, in particular, $ N \leqslant C_{G}( [ C, N ]) $. By Lemma \ref{B'CA} we have $  N^{'}  = N \leqslant C_{G}(C) $. Thus $ C \leqslant C_{G}(N) $. Since the reverse inclusion is evident we conclude that
$ C_{G}(N/Z(N))  =  C_{G}(N). $

\end{proof}

\begin{lemma}\label{NIW}

Let $G$ be a finite group and let  $ 1 = G_{0} \unlhd G_{1} \unlhd \dots \unlhd G_{n} = G $ be a chief series of $G$. Suppose that $ G_{k}/G_{k-1} $ is a chief factor of $G$ such that
\begin{enumerate}[(i)]
\item
$G_{k}/G_{k-1 } \cong W$, where $W$ is a non-abelian finite simple group;
\item
for $ i = 1, \dots , k-1 $ there do not exist subgroups $K$ and $H$ in $ \bmaut(G_{i}/G_{i-1})$ with $ K \unlhd H $ and $ H/K \cong W$.
\end{enumerate}
Then there exists a normal subgroup $ N \unlhd G $ such that $ N^{'} = N $ and $ N/Z(N) \cong W $.

\end{lemma}

\begin{proof}

We use induction on $ | G | $. If $ G_{1} = G_{k}$ then the result is trivial. If $ G_{1} \not= G_{k} $ then
\[ 1_{G/G_{1}} = G_{1}/G_{1} \unlhd G_{2}/G_{1} \unlhd \dots \unlhd G_{n}/G_{1} = G/G_{1}  \]
is a chief series of $ G/G_{1} $ that satisfies the above hypotheses. By induction there exists a subgroup $ N_{1} \unlhd G $ with $ G_{1} \leqslant N_{1} $ and $ (N_{1}/G_{1})^{'} = N_{1}/G_{1} $, and, letting $ Z_{1}/G_{1} = Z(N_{1}/G_{1}) $,  such that $ N_{1}/Z_{1} \cong (N_{1}/G_{1})/(Z_{1}/G_{1}) \cong W $.

We let $ C_{1} = C_{N_1}(G_{1})  \unlhd G $. If $ C_{1} \leqslant Z_{1}$ then, letting $ H = N_{1}/C_{1} $ and $ K = Z_{1}/C_{1}  $, we see that $H$ and $K$ can be identified with subgroups of $ \bmaut(G_{1}) $ such that $ H/K \cong W $.  This is excluded by hypothesis, so  $ C_{1} \not\leqslant Z_{1} $. Then $ 1 \not= C_{1}Z_{1}/Z_{1} \unlhd N_{1}/Z_{1} $ so, by the simplicity of $ N_{1}/Z_{1} $, we have $ C_{1}Z_{1}/Z_{1} = N_{1}/Z_{1} $ and, in particular, $ N_{1} = C_{1}Z_{1} $. It follows that
\begin{eqnarray*}
N_{1}/G_{1} & = & (N_{1}/G_{1})' \\
& = & (C_{1}Z_{1}/G_{1})^{'} \\
& = & (C_{1}^{'}G_{1}/G_{1})(Z_{1}^{'}G_{1}/G_{1})[ C_{1}, Z_{1} ]G_{1}/G_{1} \\
& = & (C_{1}^{'}G_{1}/G_{1})(1_{G/G_{1}})[ C_{1}, Z_{1} ]G_{1}/G_{1} \\
& \leqslant &  C_{1}G_{1}/G_{1}.
\end{eqnarray*}
Hence $ N_{1} = C_{1}G_{1} $.

Since $ G_{1 }$ is a chief factor of $G$ we see that either $ G_{1} =  U_{1} \times \dots \times U_{s}$, with $ U_{j} \cong U $, $ j= 1, \dots , s $, where $U$ is a non-abelian finite simple group, or $G_{1}$ is an elementary abelian p-group for some prime $p$. In the first case we have  $ Z(G_{1}) = 1 $. Thus $ N_{1} = G_{1} \times C_{1} $, whence $ C_{1} \cong C_{1}G_{1}/G_{1} \cong N_{1}/G_{1} $. We then see that  $ N = C_{1} $ is a normal subgroup of $G$ with the required properties.

In the second case $G_{1}$ is abelian, so $ N_{1} = C_{1}G_{1} = C_{1} $ and $ G_{1} \leqslant Z(N_{1}) $. Since $G_{1}$ is minimal normal in $G$ we have either $ G_{1} \cap N_{1}^{'} = 1 $ or $ G_{1} \leqslant N_{1}^{'}$. If $ G_{1} \cap N_{1}^{'} = 1 $ then $ N_{1}/G_{1} = ( N_{1}/G_{1})^{'} =  N_{1}^{'}G_{1}/G_{1} \cong N_{1}^{'}/(N_{1}^{'} \cap G_{1}) \cong N_{1}^{'} $. Hence we see that $ N = N_{1}^{'} $ is a normal subgroup with the required properties. If $ G_{1} \leqslant N_{1}^{'} $ then $ N_{1} = N_{1}^{'}G_{1} = N_{1}^{'} $. Now $ Z_{1}/G_{1} = Z(N_{1}/G_{1}) $ so $ [ N_{1}, Z_1 ] \leqslant G_{1} \leqslant Z(N_{1}) $. Thus $N_{1}$ centralises $ [ N_{1}, Z_{1} ] $ and $ [ N_{1}, Z_{1} ] $ is abelian. It follows from Lemma \ref{B'CA} that $ N_{1}^{'} = N_{1} $ centralises $Z_{1}$, so $ Z_{1} = Z(N_{1}) $. Taking $ N =N_{1} $, we then have $ N = N^{'} $ and $ N/Z(N) \cong W $, as desired.

\end{proof}

\begin{corollary}\label{N1N2}

Let  $G$ be a finite group and let $ 1 = G_{0} \unlhd G_{1} \unlhd \dots \unlhd G_{n} = G $ be a chief series of $G$. Suppose that $ G_{k_{1}}/G_{k_{1}-1} $ and $ G_{k_{2}}/G_{k_{2}-1} $ are chief factors of $G$ with $ k_{1} \not= k_{2} $ for which the following is satisfied for $ i = 1,2 $: 
\begin{enumerate}[(i)]

\item
$G_{k_{i}}/G_{k_{i}-1 } \cong W_{i} $, where $W_{i}$ is a non-abelian (finite) simple group;
\item
for $ j = 1, \dots k_{i}-1 $ there do not exist subgroups $K$ and $H$ in $ \bmaut(G_{j}/G_{j-1})$ with $ K \unlhd H $ and $ H/K \cong W_{i}$.
\end{enumerate}
Then there exist  normal subgroups $ N_{1}$ and $N_{2}$ in $G$ such that $ [ N_{1}, N_{2}] = 1 $  and such that $ N_{i}^{'} = N_{i} $ and $ N_{i}/Z(N_{i}) \cong W_{i} $ ($ i = 1,2 $).

\end{corollary}

\begin{proof}

By Lemma \ref{NIW}, there exist normal subgroups, $N_{1}$  and $N_{2}$ in $G$ such that $ N_{i}^{'} = N_{i}$ and $ N_{i}/Z(N_{i}) \cong W_{i}$ ( i = 1,2 ). Now $N_1$ and $N_2$ are normal quasisimple subgroups of $G$ and by condition $(ii)$ we see that $ N_1 / Z(N_1) \cong W_1 $ $ \not\cong W_2 \cong N_2 / Z( N_2 ) $. Thus $N_1$ and $N_2$ are distinct components of $G$. Hence by \cite{Isaacs}, Theorem 9.4 we see that $ [ N_1, N_2 ] = 1$.

\end{proof}

\begin{remark}\label{remarkN1N2}

We note that $N_{1}$ and $N_{2}$ , as in Lemma \ref{N1N2} above, have the following properties:
\begin{enumerate}[(i)]
\item
$ N_{1}N_{2}$ is the central product of $N_{1}$ and $N_{2}$.
\item
Letting  $ Z = Z(N_{1}N_{2}) $ we have, by the simplicity of $ N_{i}/Z(N_{i}) $, $ N_{i} \cap Z = Z(N_{i}) $ (i = 1, 2). Hence
\[  N_{i}Z/Z \cong N_{i}/Z(N_{i}) \cong W_{i}  \mbox{ (i = 1, 2).} \]
\item
If $  N_{1}Z/Z \cap N_{2}Z/Z \not= 1_{G/Z} $ then, by simplicity, we have $ N_{1}Z/Z \leqslant  N_{1}Z/Z \cap  N_{2}Z/Z $ and $ N_{2}Z/Z \leqslant  N_{1}Z/Z \cap  N_{2}Z/Z $ and so $ N_{1}Z/Z =  N_{2}Z/Z $, which is a contradiction (as $ W_{1} \not\cong W_{2} $). Thus
\[  N_{1}Z/Z \cap N_{2}Z/Z = 1_{G/Z} \] 
and hence
\begin{eqnarray*}
N_{1}N_{2}/Z & = & (N_{1}Z/Z)(N_{2}Z/Z) \\
& = & N_{1}Z/Z \times N_{2}Z/Z \\
& \cong & W_{1} \times W_{2 }.
\end{eqnarray*}
\end{enumerate}
We note that the above can be extended to the case where there are more than two chief factors that satisfy hypotheses (i) and (ii) of Lemma \ref{N1N2}.

\end{remark}

The remainder of this section is devoted to showing that not all groups of order $12$ or less can be embedded in a non-soluble group whose order is a divisor of $ 2^{5}.3^{3}.5.7.11.13$. We will refer to the following table, abstracted from \cite{Atlas} and \cite{AtlasV3}, which lists the non-abelian simple groups whose orders divide $ 2^{5}.3^{3}.5.7.11.13$. The outer automorphism group $ \bmaut(G)/\bminn(G) $ and the Schur multiplier $M$ of each group is also listed.

\begin{table}[H]

\begin{center}

\begin{tabular}{ccccc}
\hline
Order & Factorised Order&Group & $M$ & $\bmaut(G)/\bminn(G)$ \\
\hline
\rule{0pt}{3ex}

$60$&$2^{2}.3.5$&$A_{5}$&$C_{2}$&$C_{2}$\\
$168$&$2^{3}.3.7$& $ PSL(3, 2) \cong GL(3, 2) $&$C_{2}$&$C_{2}$\\
$360$&$2^{3}.3^{2}.5$&$A_{6}$&$C_{6}$&$C_{2} \times C_{2}$\\
$504$&$2^{3}.3^{2}.7$&$PSL( 2, 8)$&1&$C_{3}$\\
$660$&$2^{2}.3.5.11$&$PSL(2, 11)$&$C_{2}$&$C_{2}$\\
$1092$&$2^{2}.3.7.13$&$PSL(2, 13)$&$C_{2}$&$C_{2}$\\
$2520$&$2^{3}.3^{2}.5.7$&$A_{7}$&$C_{6}$&$C_{2}$\\
$5616$&$2^{4}.3^{3}.13$&$PSL(3, 3)$&$1$&$C_{2}$\\
$6048$&$2^{5}.3^{3}.7$&$U_{3}(3)$&$1$&$C_{2}$\\
$7920$&$2^{4}.3^{2}.5.11$&$M_{11}$&$1$&$1$\\
$9828$&$2^{2}.3^{3}.7.13$&$PSL(2, 27)$&$C_{2}$&$C_{6}$\\
\hline
\end{tabular}

\end{center}

\caption{Non-abelian finite simple groups whose order is a divisor of $ 2^{5}.3^{3}.5.7.11.13$}\label{simplegroups}

\end{table}

Lemmas \ref{2CF}-\ref{C9A6} and Corollary \ref{GDxy} now deal with specific cases of groups in which not all groups of order 8 or 9 can be embedded.

\begin{lemma}\label{2CF}

Let $G$ be a finite group whose order is a divisor of $ 2^{5}.3^{3}.5.7.11.13$. Suppose that $ 1 = G_{0} \unlhd G_{1} \unlhd \dots \unlhd G_{n} = G $ is a chief series of $G$ for which at least two distinct chief factors are not soluble. Then not all groups of order $8$ can be embedded in $G$.

\end{lemma}

\begin{proof}

We let $ G_{k_{1}}/G_{k_{1}-1}$ and $ G_{k_{2}}/G_{k_{2}-1}$ for $ k_{1} \not= k_{2}$ be two distinct non-soluble chief factors of the above chief series. For $ i = 1, 2 $ we have
\[  G_{k_{i}}/G_{k_{i}-1} \cong W_{i_{1}} \times \dots \times W_{i_{n_{i} }}, \]
where $ W_{i_{j}} \cong W_{i}$ for $ j = 1, \dots, n_{i}$ and $W_{i}$ is a non-abelian simple group ($ i = 1, 2 $). As all $\{2, 3\}$-groups are soluble by Burnside's Theorem, there must be a prime $ p_{i} \in \{5, 7, 11, 13\} $ such that $p_{i}$ is a divisor of $|W_{i}|$ ($ i= 1, 2 $). If $ n_{i} \geq 2 $ then $ p_i^2$ is a divisor of $ |G_{k_{i}}/G_{k_{i}-1}| $ and hence of $ |G|$. As this is excluded by hypothesis, we have $ n_{i} = 1 $ (for $ i = 1, 2 $), so $ G_{k_{i}}/G_{k_{i}-1} \cong W_{i} $ is a non-abelian simple group for $i = 1, 2$. We see that $ |W_{1}| |W_{2}| = |G_{k_{1}}/G_{k_{1}-1}||G_{k_{2}}/G_{k_{2}-1}|$ is a divisor of $|G|$. In particular $ p_{1}p_{2} $ is a divisor of $|G|$, so $ p_{1} \not= p_{2}$ and $W_{1}$  and $W_{2}$ are not isomorphic.

From Table \ref{simplegroups} we see that $2^2.3.p_{i}$ is a divisor of $|W_{i}|$ (i=1, 2). As $ |W_{1}| |W_{2}|$ is a divisor of $|G|$ it follows that the remaining chief factors have orders dividing  $ \dfrac{|G|}{2^{4}.3^{2}.p_{1}.p_{2}} = 2.3.r_{1}.r_{2}$, where $ \{  r_{1}, r_{2} \} = \{ 5, 7, 11, 13 \} \backslash \{p_{1}, p_{2} \} $. But there are no groups in Table \ref{simplegroups} whose orders are factors of $ 2.3.r_{1}.r_{2} $. Thus the remaining chief factors, if any, must be of order 2, 3, 5, 7, 11 or 13. As these are all primes we see, in particular, that the automorphism groups of the remaining chief factors are cyclic.

We further see from Table \ref{simplegroups} that $ p_{1} \nmid  | \bmaut(W_{2})|  $ and $ p_{2} \nmid  | \bmaut(W_{1})|  $. Therefore the conditions for Corollary \ref{N1N2} are satisfied, so there exist normal subgroups $N_{1}$ and $N_{2}$ in $G$ such that $ [N_{1}, N_{2}] = 1 $ and, for i = 1, 2, $ N_{i}^{'} = N_{i} $ and $ N_{i}/Z(N_{i}) \cong W_{i} $. By Remark \ref{remarkN1N2} we have
\begin{eqnarray*}
N_{1}N_{2}/Z(N_{1}N_{2}) & = & N_{1}Z(N_{1}N_{2})/Z(N_{1}N_{2}) \times N_{2}Z(N_{1}N_{2})/Z(N_{1}N_{2})\\
& \cong & W_{1} \times W_{2}.
\end{eqnarray*}
We further note that, since $ 4 \mid |W_{i}|$  ( for $ i =1 ,2 $), we have $ 16 \mid  |N_{1}N_{2}/Z(N_{1}N_{2})| $.

We let $S_{i}$ be a Sylow 2-subgroup of $N_{i}$ (for $ i = 1, 2 $). Since $ [N_{1}, N_{2}] = 1 $ we have $ [S_{1}, S_{2}] = 1 $ so $ S_{1}S_{2} $ is a group, and thus a Sylow 2-subgroup of $ N_{1}N_{2} $. Since $ N_{1}N_{2} $ is a central product we have
\[  S_{1} \cap S_{2} \leqslant Z(N_{1}N_{2}) \cap S_{1}S_{2}.\]
Now $ S_{i}Z(N_{1}N_{2})/ $ $Z(N_{1}N_{2}) $ ($ \cong S_{i}/(S_{i} \cap Z(N_{1}N_{2})) $) is a Sylow 2-subgroup of $ N_{i}Z(N_{1}N_{2})/ $$Z(N_{1}N_{2}) $ $ (\cong W_{i})$, for $  i = 1, 2$. Thus $4$ is a factor of $ |S_{i}/(S_{i} \cap Z(N_{1}N_{2}))| $ and, since $ S_{1} \cap S_{2} \leqslant N_{1} \cap N_{2} $ $ \leqslant Z(N_{1}N_{2}) $, we see that $4$ is also a factor of $ | S_{i}/(S_{1} \cap S_{2})| $. Therefore $ |S_{i}| \geq 4|S_{1} \cap S_{2}|$ and 
\begin{eqnarray*}
|S_{1}S_{2}| & = & \dfrac{ |S_{1}||S_{2}| }{ |S_{1} \cap S_{2}|}\\
& \geq & \dfrac{4| S_{1} \cap S_{2} |.4| S_{1} \cap S_{2}|}{ | S_{1} \cap S_{2}| }\\
& = & 16|  S_{1} \cap S_{2} |.
\end{eqnarray*}
Furthermore since $ 2^{5} \top |G| $, we have $ |S_{1}S_{2}| \leq 32 $, and so $ | S_{1} \cap S_{2} | \leq 2 $.

If $ | S_{1} \cap S_{2} | = 2 $ then $ | S_{1}S_{2} | = 32 $ and $|S_{i}| = 8 $ ($ i = 1, 2 $). In particular $ S_{1}S_{2} $ is a Sylow 2-subgroup of $G$. If both $S_{1}$ and $S_{2}$ have exponent at most $4$ then the central product $S_{1}S_{2}$ has exponent at most $4$. In this case $C_{8}$ cannot be embedded in $G$. If, say, $S_{1}$ has exponent $8$ then $ S_{1} \cong C_{8} $ and, in particular, $ S_{1} \leqslant Z(S_{1}S_{2}) $. If $ C_{2} \times C_{2} \times C_{2} $ can be embedded in $G$ then $ S_{1}S_{2}$ will contain a subgroup $ V \cong C_{2} \times C_{2} \times C_{2} $. By comparison of orders,  $ S_{1}S_{2} $ is then the product of the central subgroup $S_{1}$ and the abelian subgroup $V$, so $ S_{1}S_{2} $ is abelian. Hence the non-abelian groups of order $8$ cannot be embedded in $G$.

If $ |S_{1} \cap S_{2}| = 1$ then $ S_{1}S_{2} =  S_{1} \times S_{2} $ and $ |S_{1}S_{2}| =  |S_{1}||S_{2}| $. If, say, $ |S_{1}| = 8 $ then, since $ |S_{i}| \geq 4$ (i = 1, 2), we have $ |S_{2}| = 4 $, so $S_{2}$ is abelian. Since $ |S_1| = 8 $ there exists a subgroup $ K \leqslant S_{1} $ such that $ |K| = 4 $. Then $K$ is abelian and $KS_{2} = K \times S_2 $ is an abelian subgroup of order $16$ and exponent at most $4$ in $G$. We can thus apply Lemma \ref{Habex4} to see that not all groups of order $8$ can be embedded in $G$.

Finally, if $ |S_{1} \cap S_{2}| = 1$ and $|S_{1}| = |S_{2}| = 4 $, then $ S_1S_2$ is abelian of order 16 and exponent at most $4$. We may again apply Lemma \ref{Habex4} to see that not all groups of order $8$ can be embedded in $G$.

\end{proof}

\begin{lemma}\label{HDx}

Let $H$ be a group of order 16 such that $H = D\langle y \rangle $ where
\begin{enumerate}[(i)]
\item
$ D \cong D_4 $ and $ \langle y \rangle \cong C_2 $;
\item
$ D \cap \langle y \rangle = 1 $;
\item
$ C_H(D) = Z(D)$.
\end{enumerate}
Then $H$ has a unique cyclic subgroup of order $4$. In particular, the quaternion group of order $8$ cannot be embedded in $H$.

\end{lemma}

\begin{proof}
Since $ |H : D| = 2 $ we see that $ D \unlhd H $. We let $D$ be as follows:
\[ D = \langle \ a, \ b \mid a^2 = b^4 = 1, \ aba = b^{-1} \ \rangle. \]
Since $ \bmaut(D_4) \cong D_4 $ and $ C_H(D) = Z(D)$ ($\cong C_2$) we have (by comparison of orders):
\[ H/C_H(D) = H/Z(D) = H/\langle b^2 \rangle \cong D_4. \]
Now $ \langle b \rangle $ is characteristic in $D$ and so is normal in $H$. Both $y$ and $a$ have order $2$ and $ y \not\in \langle a, b \rangle $, so $ y \langle b \rangle $ and $ a \langle b \rangle $ are distinct elements of order 2 in $ H/\langle b\rangle $. In addition $ |H/\langle b \rangle| = 4 $, so $H/\langle b \rangle $ is abelian and hence elementary abelian. Thus $ \Phi(H) \leqslant \langle b \rangle $.
In addition we have $ \langle b^2 \rangle = \Phi( \langle b \rangle ) \leqslant \Phi(H) $, so, if $ \Phi(H) \neq \langle b \rangle $ we must have $ \Phi(H) = \langle b^2 \rangle $. Then $ H/\langle b^2 \rangle $ is elementary abelian, in contradiction to  $ H/\langle b^2 \rangle \cong D_4 $. Hence
\[ \Phi(H) = \langle b \rangle. \]
Since $ H = \langle \ a, \ b, \ y \ \rangle $ we then have $ H = \langle \ a, \ y \ \rangle $ and, since $ \Phi(H) $ is cyclic, we see that $ H^{'} = \langle [a, y] \rangle $. If $ [a, y] \in \langle b^2 \rangle $ then $ H / \langle b^2 \rangle $ is abelian, which contradicts $ H/ \langle b^2 \rangle \cong D_4 $. Therefore $ [a, y] \in \langle b \rangle \backslash \langle b^2 \rangle $, so either $ [a, y] = b $ or $ [a, y] = b^{-1} $. If $ [a, y] = b $ then $ a^y = ab $. If $ [a, y] = b^{-1} $ we likewise have $ a^y = ab^{-1} $.

Since $ \langle b \rangle \unlhd H $ and $ \langle b \rangle \cong C_4 $ we have either $ b^y = b $ or $ b^y = b^{-1} $. If $ b^y = b $
 and $ a^y = ab $ then
\[  a = a^{y^2} =(a^y)^y = (ab)^y = a^yb^y = abb = ab^2, \]
which is a contradiction since $ b^2 \neq 1 $. Similarly if $ b^y = b $ and $ a^y = ab^{-1} $ we have
\[  a = a^{y^2} = ab^{-2}, \]
and once more a contradiction ensues. We thus conclude
\[  b^y = b^{-1}. \]

We now determine the elements of order $4$ in $H$. If $ h \in H $  and $ o(h) = 4 $, then $ h^2 $ is an element of order $2$ in $ \Phi(H) = \langle b \rangle $. Since $ \langle b^2 \rangle $ is the unique element of order $2$ in $ \langle b \rangle $ we see that $ o(h) = 4 $ if and only if  $ h^2 = b^2 $. We let
\[  h = a^{\alpha}b^{\beta}y^{\gamma},\]
 for suitable $ \alpha, \ \beta \ \mbox{and} \ \gamma $. If $ \gamma \equiv 0 $ (mod 2) then $ h = a^{\alpha}b^{\beta} \in \langle \ a, \ b \ \rangle \cong D_4 $, so $ h \in \langle b \rangle $. If $ \gamma \equiv 1 $ (mod 2) and $ \alpha \equiv 0 $ (mod 2) then $ h = b^{\beta}y $. In this case $ o(h) = 4 $ if and only if 
\begin{eqnarray*}
b^2 & = & h^2\\
& = & (b^{\beta}y)^2\\
& = & (b^{\beta}y)(b^{\beta}y)\\
& = & b^{\beta}(b^{\beta})^y\\
& = & b^{\beta}b^{-\beta} = 1,
\end{eqnarray*}
which is a contradiction.
If $ \gamma \equiv 1 $ (mod 2) and $ \alpha \equiv 1 $ (mod 2) then $ h = ab^{\beta}y $, so $ o(h) = 4 $ if and only if
\begin{eqnarray*}
b^2 &  = & (ab^{\beta}y)^2\\
& = & ab^{\beta}yab^{\beta}y\\
& = & ab^{\beta}a^y(b^{\beta})^y\\
& = & ab^{\beta}a^yb^{-\beta}.
\end{eqnarray*}
If, in addition,  $ a^y = ab $ then
\begin{eqnarray*}
b^2 & = & ab^{\beta}abb^{-\beta}\\
& = & b^{-\beta}bb^{-\beta}\\
& = & bb^{-2\beta}
\end{eqnarray*}
and the contradiction $ b = b^2b^{2\beta} \in \langle b^2 \rangle $ ensues. Finally if $ \gamma \equiv  \alpha \equiv 1 $ (mod 2) and $ a^y =ab^{-1} $, then $ o(h) = 4 $ if and only if
\begin{eqnarray*}
b^2 & = & ab^{\beta}ab^{-1}b^{-\beta}\\
& = & b^{-\beta}b^{-1}b^{-\beta}\\
& = & b^{-1}b^{-2\beta},
\end{eqnarray*}
so $ b^{-1} = b^2b^{2\beta} \in \langle b^2 \rangle $, which is again a contradiction.

Thus the only elements of order 4 in $H$ are $b$ and $b^{-1}$, so we conclude that $ \langle b\rangle $ is the unique cyclic subgroup of order $4$ in $H$.

\end{proof}

\begin{corollary}\label{GDxy}

Let $G$ be a group of order 32 such that $ G = ( D \times \langle x \rangle ) \langle y \rangle $, where
\begin{enumerate}[(i)]

\item
$ \langle x \rangle \unlhd G $;

\item
$ \langle x \rangle \cong C_2 $;

\item
$D \cong D_4 $;

\item
$ y^2 \in \langle x \rangle $.
\end{enumerate}
If $Q_2$ can be embedded in $G$ then $G$ has exponent 4 and, in particular, $C_8 $ cannot be embedded in $G$.

\end{corollary}

\begin{proof}

We let $ Q \leqslant G $ be such that $ Q \cong Q_2 $. Then $ | Q \cap ( D \times \langle x \rangle ) | \geq 4 $, as otherwise $ |G| \geq | Q( D \times \langle x \rangle ) | \geq 64$. Now $ D \times \langle x \rangle \cong D_4 \times C_2 $ has exactly four elements of order 4 and $Q_2$ has six elements of order 4. Hence $ Q \not\leqslant D \times \langle x \rangle $, so $ | Q \cap ( D \times \langle x \rangle ) | = 4 $. Since every proper subgroup of $Q_2$ is cyclic we let $ Q \cap ( D \times \langle x \rangle ) = \langle x_1 \rangle $, where $ o(x_1) = 4$. As the square of every element in $ D_4 \times C_2 $ is contained in $ \Phi ( D_4 \times C_2 ) = D_4^{'} $, we have $ \langle x_1^2 \rangle = D^{'} \cong C_2 $. It follows that $ Q \cap \langle x \rangle = 1 $. Thus $ Q \langle x \rangle / \langle x \rangle \cong Q $, so $Q_2$ can be embedded in $ G/ \langle x \rangle $.

Now $ D\langle x \rangle/\langle x \rangle \cong D_4 $ and $ \langle \ y, \ x \ \rangle / \langle x \rangle \cong C_2 $. Hence, since $ | G / \langle x \rangle | = 16 $, we see, by comparison of orders, that $ D \langle x \rangle / \langle x \rangle \cap \langle \ y, \ x \ \rangle / \langle x \rangle = 1_{ G / \langle x \rangle } $. Since $Q_2$ can be embedded in $ G/ \langle x \rangle $, we see by Lemma \ref{HDx} that $ Z(D \langle x \rangle / \langle x \rangle ) $ is a proper subgroup of $ C_{ G / \langle x \rangle }( D \langle x \rangle / \langle x \rangle ) $. Now $ | D \langle x \rangle / \langle x \rangle \cap C_{ G / \langle x \rangle }( D \langle x \rangle / \langle x \rangle ) | = | Z( D \langle x \rangle / \langle x \rangle )| = 2 $ so, by comparison of orders, we have $ | C_{ G / \langle x \rangle }( D \langle x \rangle / \langle x \rangle ) |  = 4 $ and see that $ G / \langle x \rangle $ is the central product of $ D \langle x \rangle / \langle x \rangle $  and $ C_{ G / \langle x \rangle }( D \langle x \rangle / \langle x \rangle ) $. But $ D \langle x \rangle / \langle x \rangle $ has exponent 4 and $ C_{ G / \langle x \rangle }( D \langle x \rangle / \langle x \rangle ) $ has exponent 2 or 4. Hence we conclude that $ G / \langle x \rangle $ has exponent 4.

If $G$ does not have exponent 4, then we let $ g \in G $ be such that $ o(g) = 8 $. We have $ g^2 \in D \times \langle x \rangle $ (since $ | G : D \times \langle x \rangle | = 2 $) so, as above, $ \langle g^4 \rangle = \Phi( D \times \langle x \rangle) = D^{'}$. But, since $ G/ \langle x \rangle $ has exponent 4, we also have $ g^4 \in \langle x \rangle $ and the contradiction $ 1 \neq  g^4 \in D^{'} \cap \langle x \rangle = 1 $ arises.

\end{proof}

\begin{lemma}\label{GA4N}

Let $G$ be a finite group such that
\begin{enumerate}[(i)]

\item
$ 2^5 \top |G| $;

\item
$G$ has a subgroup, $U$, with $ U \cong A_4 $;

\item
$G$ has a normal subgroup, $N$, with the following properties:

\begin{enumerate}

\item
$ N = N^{'}  $;

\item
$ Z(N/Z(N)) = 1 $;

\item
$ 4 \mid |N/Z(N)| $;

\item
either $ 2 \nd |Z(N)| $ or $ 2 \top |Z(N)| $;

\item
$ 16 \nmid |\bmaut(N/Z(N))| $;

\item
$ 3 \nd |\bmaut(N/Z(N))/\bminn(N/Z(N))| $.

\end{enumerate}

\end{enumerate}
Then not all groups of order $8$ can be embedded in $G$.

\end{lemma}

\begin{proof}

We let $ Z(N) = A \times B $, where $A$ is a 2-group and $B$ is a $2^{'}$-group. Then $ B \unlhd G $ and $ U \cap B \leqslant O_{2^{'}}( U ) = 1 $. In addition the Sylow 2-subgroups of $ G/B $ are isomorphic to those of $G$ and, since $ Z( N / Z(N) ) = 1 $, we see that $ Z(N / B) = Z(N) / B $ and it follows that $ Z( (N / B ) / Z( N / B ) ) = Z( (N / B) / ( Z(N) / B ) ) \cong Z( N / Z(N) ) = 1 $. Hence we may work "modulo $B$" and assume, without loss of generality, that $ B = 1 $.

We let $ C/Z(N) = C_{G/Z(N)}(N/Z(N)) $. We have $ C \unlhd G $ and, by $(iii) (b)$, $ C \cap N = Z(N) $. By Corollary \ref{CGN} we have $ C = C_G(N) $. We let $S_1$ be a Sylow 2-subgroup of $N$ and let $S_2$ be a Sylow 2-subgroup of $C$. Then $ S_1S_2 $ is a central product, with $ S_1 \cap S_2 = Z(N) $. By $(iii) (b) $ we further have
\[ N/Z(N) = N/(N \cap C) \cong NC/C \cong \bminn(N/Z(N)). \]
Hence $ G/NC \cong (G/C)/(NC/C) $ is isomorphic to a subgroup of $ \bmaut(N/Z(N))/$$\bminn(N/Z(N)) $. Since $ 3 \nmid | \bmaut(N/Z(N))/\bminn(N/Z(N)| $ we have $ 3 \nmid | G/NC| $ so $ O^{3^{'}}(G) \leqslant NC $. In particular $ U \leqslant NC$. Furthermore, since $ 2^5 \top |G| $ and $ 16 \nmid |\bmaut(N/Z(N))| $, we have $ 4 \mid |C| $.  In addition, since $ 4 \mid |N/Z(N)| $, we see that either $ 2^2 \top |C| $ or $ 2^3 \top |C| $.

We first suppose that $ 2^2 \top |C| $. Then $ |S_2| = 4$. If $ Z(N) = 1 $ then $ S_1S_2 = S_1 \times S_2 $. Since $ 4 \mid |N/Z(N)| $ we see that $S_1$ has a subgroup of order $4$, $H$ say. Then $ H  \times S_1 $ has order $2^4$ and exponent at most $4$. As $H$ and $S_1$ are both abelian, we apply Lemma \ref{Habex4} to see that not all groups of order $8$ can be embedded in $G$. If $ Z(N) \not= 1 $, then by ($iii$)($d$) we have $ Z(N) \cong C_2 $, and so $ 2 \top |C/Z(N)| $. If $ U^{'} \cap N = 1 $ then
\begin{eqnarray*}
U^{'} & \cong & U^{'}N/N\\
& \cong & ( U^{'}N/Z(N))/(N/Z(N))\\
& \leqslant & (NC/Z(N))/(N/Z(N))\\
& \cong & NC/N\\
& \cong & C/(N \cap C) = C/Z(N).
\end{eqnarray*}
But $ U^{'} \cong C_2 \times C_2 $ and $ 2 \top |C/Z(N)| $, so this cannot be. Thus $ U^{'} \cap N \not= 1 $. But $ U^{'} \cap N \unlhd U \cong A_4 $ so, by minimal normality, $ U^{'} \cap N = U^{'} $. It follows that $ U^{'}Z(N) \leqslant N$. Since $ U^{'} \cap Z(N) \unlhd U $ and $ Z(N) \cong C_2 $, it again follows by the minimal normality of $U^{' }$ in $U$, that $ U^{'} \cap Z(N) = 1 $. Hence $ U^{'}Z(N) = U^{'} \times Z(N) \cong C_2 \times C_2 \times C_2 $ . Now $ |S_2| = 4 $, so $S_2$ is abelian. Since $ S_2 \cap (U^{'} Z(N)) \leqslant S_1 \cap S_2 = Z(N) \cong C_2 $, we see that $ S_2(U^{'}Z(N)) $ is a central product of abelian groups, has order 16, and has exponent at most 4. By Lemma \ref{Habex4} we see that not all groups of order $8$ can be embedded in $G$.

We may now assume that $ 2^3\top |C| $. Then $ |S_2| = 8 $. If $ Z(N) = 1 $ then $ S_1S_2 = S_1 \times S_2 $. From ($iii$)($c$) we have $ |S_1| \geq 4 $, so $ |S_1| = 4 $ and we see that $ S_1 \times S_2 $ is a Sylow 2-subgroup of $G$. If $C_8$ can be embedded in $G$ then $S_2$ must have exponent $8$, so $ S_2 \cong C_8 $. Then $ S = S_1 \times S_2 $ is abelian, so the non-abelian groups of order $8$ cannot be emebdded in $G$. If $Z(N) \neq 1 $, then $ Z(N) \cong C_2 $. As in the preceding paragraph, if $ U^{'} \cap N = 1 $ and $ U^{'} \cap C = 1$, then $ C_2 \times C_2 $ can be embedded in $ C/Z(N) $ and, similarly, in $ N/Z(N) $. By comparison of orders, we have
\begin{eqnarray*}
S_1S_2/Z(N) & = & S_1/Z(N) \times S_2/Z(N)\\
& \cong & C_2 \times C_2 \times C_2 \times C_2.
\end{eqnarray*}
Since $Z(N) \cong C_2 $ we then see that the Sylow 2-subgroups of $G$ have exponent at most $4$, so $C_8$ cannot be embedded in $G$.

Thus we may assume that either $ U^{'} \cap C \neq 1 $ or $ U^{'} \cap N \neq 1 $. Suppose that $ U^{'} \cap C \neq 1 $. Then, by minimal normality, we have $ U^{'} \leqslant C $ and, as above, $ U^{'}Z(N) = U^{'} \times Z(N) \cong C_2 \times  C_2 \times  C_2 $. Hence $ S_2 = U^{'}Z(N) $. As above, letting $ H \leqslant S_1 $ be such that $ |H| = 4 $, we see that $ HS_2 $ is an abelian subgroup of order $16$ and exponent at most $4$. Again we apply Lemma \ref{Habex4} to see that not all groups of order $8$ can be embedded in $G$. Finally, if $ U^{'} \cap N \neq 1 $ we likewise see that $ U^{'}Z(N) \leqslant N $ and $ U^{'}Z(N) \cong C_2 \times C_2 \times C_2 $. Taking the (central) product of $ U^{'}Z(N) $ with a subgroup of order $4$ in $S_2$, we once more apply Lemma \ref{Habex4} to conclude that not all groups of order $8$ can be embedded in $G$.

\end{proof}

\begin{lemma}\label{C9A6}

Let $G$ be a finite group such that $ 3^3 \top |G| $ and let $H$ be a subgroup of $G$ such that 
\begin{enumerate}[(i)]
\item
$ 3 \mid | Z(H) | $; 
\item
$ H/Z(H) $ has a subgroup isomorphic to $A_6$.
\end{enumerate}
Then $C_9$ cannot be embedded in $G$.
\end{lemma}

\begin{proof}

We let $ K \leqslant H $ be such that $ Z(H) \leqslant K$ and $ K/Z(H) \cong A_6 $. Since $ 9 \mid |A_6| $ we see that $27$ is a divisor of $ |K| $. Thus the Sylow 3-subgroups of $G$ are isomorphic to those of $K$. We suppose that $C_9$ can be embedded in $G$ and let $ U \leqslant K $ be such that $ U \cong C_9 $. 

We identify $ K/Z(H) $ with $A_6$ and assume, without loss of generality, that $ Z(H) \cong C_3 $. As the Sylow 3-subgroups of $A_6$ are isomorphic to $ C_3 \times C_3 $ we see that $ Z(H) \leqslant U $. Consideration of the possible cycle types of order three in $A_6$ shows that we may identify $ U/Z(H) $ with either $ \langle \ (123) \ \rangle $ or $ \langle \ (123)(456) \ \rangle $. Now $ \langle \ (123), \ (12)(45) \ \rangle $ and $ \langle \ (123)(456), \ (12)(45) \ \rangle $ are subgroups of $A_6$, both of which are isomorphic to $S_3$. Hence there exists an element $ x \in K $ such that $ \langle \ x, \ U \ \rangle / Z(H) \cong S_3  $. Without loss of generality, we may assume that $ o(x) = 2 $. It follows that conjugation by $x$ induces an automorphism of order $2$ on $ U/Z(H) = U / \Phi(U) $. Then conjugation by $x$ induces an automorphism of order $2$ on $U$, so we may assume that conjugation by $x$ inverts every element of $U$. It follows that $x$ does not centralize $ \Phi(U) = Z(H) $, which is a contradiction. We conclude that $C_9$ cannot be embedded in $G$.

\end{proof}

Our next lemma provides some information that will be useful in the proof of Proposition \ref{12embed}, the main result of this section.

\begin{lemma}\label{Inny}

There exists an element $ y \in \bmaut(GL(3, 2)) $ with $ o(y) = 2 $, and such that $ \bmaut(GL(3, 2)) = \bminn(GL(3, 2))\langle y \rangle $.

\end{lemma}

\begin{proof}

We let $ W = GL(3, 2) $ and identify $W$ with $ \bminn(GL(3, 2)) $. From Table \ref{simplegroups} we see that $ |\bmaut(W) : W| = 2 $. We let $R$ be a Sylow 7-subgroup of W. Since $ GL(3, 2) $ is simple we see that $R$ is not normal in $W$. Hence, by Sylow's Theorems, we have $ |W : N_W(R)| = 8 $, whence $ |N_W(R)| = 21 $. In particular there is a non-trivial homomorphism from $W$ into $S_8$. But $W$ is simple, so we see that $W$ can be embedded in $S_8$.

We let $P$ be a Sylow 3-subgroup of $N_W(R)$. Then $P$ is also a Sylow 3-subgroup of $W$, and we have $ P \cong C_3 $ and $ N_W(R) = PR $. Now $N_W(R)$ cannot be abelian, as otherwise  $ N_W(R) \cong C_{21 } $ and $W$ would possess an element of order $21$, in contradiction to the fact that $W$ can be embedded in $S_8$ (which has no elements of order $21$). Thus $P$ normalises, but does not centralise, $R$. It follows that $P$ has precisely $7$ conjugates in $ N_W(R) $.

Then, by Sylow's Theorems, $P$ has either $7$ or $28$ conjugates in $W$. If $P$ has $7$ conjugates in $W$ then $ O^{3^{'}}(W) = \langle \ P^w \mid w \in W \ \rangle = \langle \  P^r \mid r \in R \ \rangle = N_W(R) $. But then $ N_W(R) = O^{3^{'}}(W) \unlhd W $, which contradicts the simplicity of $W$. Thus $P$ has $28$ conjugates in $W$, so $ |W : N_W(P)| = 28 $ and $ |N_W(P)| = 6 $.

Now 
$ S_3 \cong \langle \  \left(  \begin{array}{rrr}
0 &1 & 0 \\
-1 & -1 & 0 \\
0 & 0 & 1 \end{array} \right) , \  
\left(  \begin{array}{rrr}
0 &1 & 0 \\
1 & 0 & 0 \\
0 & 0 & 1 \end{array} \right) \ \rangle 
\leqslant GL(3, 2) $.
Thus the normalisor of the Sylow 3-subgroup 
$ \langle \ \left( \begin{array}{rrr}
0 & 1 & 0 \\
-1 & -1 & 0 \\
0 & 0 & 1 \end{array} \right)  \ \rangle $ in $ GL(3, 2)$ is isomorphic to $S_3$, and it follows that $ N_W(P) \cong S_3$. We let $ \langle x \rangle $ be a Sylow 2-subgroup of $N_W(P)$ and have $ \langle x \rangle \cong C_2 $ and $ N_W(P) = \langle x \rangle P $. 

Since $W$ is identified with $ \bminn(W) $, the Frattini argument shows that $ \bmaut(W) = WN_{\bmaut(W)}(P)$. Hence
\begin{eqnarray*}
|N_{\bmaut(W)}(P) : N_W(P)| & = & |N_{\bmaut(W)}(P) : W \cap N_{\bmaut(W)}(P)|\\
& = &|WN_{\bmaut(W)}(P) : W|\\
& = & |\bmaut(W) : W| = 2 .
\end{eqnarray*}
Thus $N_W(P)$ has index $2$ in $N_{\bmaut(W)}(P)$, so $ |N_{\bmaut(W)}(P) | = 12 $. We let $S$ be a Sylow 2-subgroup of $N_{\bmaut(W)}(P) $ such that $ \langle x \rangle \leqslant S $. Then $ |S| = 4 $. Since $ \langle x \rangle P \cong S_3 $ we see that $x$ does not centralise $P$, so $ |C_S(P | = 2 $. Thus $ S = \langle x \rangle \times C_S(P) $. Now, $ C_S(P) \cap  W \leqslant S \cap C_W(P) = S \cap P = 1$. Letting $ C_S(P) = \langle y \rangle $, we thus have $ o(y) = 2 $ and $  y \notin W $. Hence we conclude that $ \bmaut(W) = W\langle y \rangle $, with $ y^2 = 1$, as desired.

\end{proof}

\begin{corollary}\label{SInny}

Let $S$ be a Sylow 2-subgroup of $ \bmaut(GL(3, 2))$. Then there exists an element $ y \in \bmaut(GL(3, 2)) $ such that $ o(y) = 2 $ and 
\[ S = (S \cap \bminn(GL(3 ,2)))\langle y \rangle , \]
where $ S \cap \bminn(GL(3, 2)) \cong D_4 $ and $ (S \cap \bminn(GL(3 ,2))) \cap \langle y \rangle = 1 $.

\end{corollary}

\begin{proof}

Let $ \langle y \rangle $ be as in Lemma \ref{Inny} above and let $S$ be a Sylow 2-subgroup of $\bmaut(GL(3,2 ))$ with $ y \in S $. Then $ S \cap \bminn(GL(3, 2)) $ is a Sylow 2-subgroup of $ \bminn( GL(3, 2))$. Now $ H = \{ \ \left( 
\begin{array}{rrr} 
1 & 0 & 0\\
a & 1 & 0\\
b & c & 1 \end{array} \right) \ \mid \ a, \ b, \ c \in GF(2) \  \} 
$ is a Sylow 2-subgroup of $ GL (3, 2) $ and $ H \cong D_4 $. Thus $ S \cap \bminn(GL(3, 2)) \cong D_4 $ . Since $ o(y) = 2 $  and $ y\notin \bminn (GL(3, 2)) $, it follows, by comparison of orders, that $ S = (S \cap \bminn(GL(3, 2)))\langle y \rangle $ and $ (S \cap \bminn(GL(3 ,2))) \cap \langle y \rangle = 1 $.

\end{proof}

\begin{proposition}\label{12embed}

Let $G$ be a finite group whose order is a divisor of $ 2^5.3^3.5.7.11.13 $. Suppose that $ 1 = G_0 \unlhd G_1 \unlhd \dots \unlhd G_n = G $ is a chief series of $G$ for which exactly one chief factor, $ G_k/G_{k-1} $, is not soluble. Then not all groups of order $12$ or less can be embedded in $G$.

\end{proposition}

\begin{proof}

As in the proof of Lemma \ref{2CF} we see that $ G_k/G_{k-1} \cong W $, where $W$ is a non-abelian finite simple group. By Lemma \ref{2CF} we may also assume that the remaining chief factors are elementary abelian. Consideration of Table \ref{simplegroups} shows that $ 2^2 \mid |W| $ and that, where $ 2^2 \top |W| $, then either  5, 11  or 13  is also a divisor of $|W|$. Now, where $ 2^2 \top |W| $, the maximum possible rank of a 2-chief factor is $ r = 3 $. Thus if $ G_i/G_{i-1} $ is a 2-chief factor of rank $3$ then $|W|$ is not a divisor of $ |\bmaut(G_i/G_{i-1})|$  ($ = |GL(3, 2)| $). In addition $3$ is also a divisor of $|W|$, so the maximum possible rank of a 3-chief factor is $ r = 2 $, and we note that $ \bmaut(C_3 \times C_3) \cong GL(2, 3) $ is soluble.

Thus the conditions for Lemma \ref{NIW} are satisfied, so there exists a normal subgroup $ N \unlhd G $ such that $ N = N^{'} $ and $ N/Z(N) \cong W $. We let $ C/ Z(N) = C_{G/Z(N)} (N/Z(N)) $. By Corollary \ref{CGN}, $ C = C_G(N) $, so $NC$ is the central product of $N$ and $C$. In particular $ C \cap N = Z(N) $. Letting $S_1$, $P_1$, $S_2$ and $P_2$ be Sylow 2- and 3-subgroups of $N$ and $C$ respectively, we see that the central products $S_1S_2$ and $P_1P_2$ are Sylow 2- and 3-subgroups respectively of $NC$.

We proceed to deal with the possible isomorphism types for $W$ in ascending order of $|W|$. We let $M$ denote the Schur multiplier of $W$ and will make extensive use of the fact (as noted in \cite{Atlas}, p.xvii) that $Z(N)$ will always be isomorphic to a factor group of $M$.

\begin{case}

$ |W| = 60 = 2^2.3.5 ; \ W \cong A_5  ; \  M \cong C_2 ; \  \bmaut(W)/\bminn(W) \cong C_2 $.

\end{case}

In this case Z(N) is isomorphic to a factor group of $ M \cong C_2 $. Thus either $ 2 \nd |Z(N)| $ or $ 2 \top |Z(N)| $. In addition we see that $ 16 \nmid |\bmaut(N/Z(N))| $ and $ 3 \nmid |\bmaut(N/Z(N))/\bminn(N/Z(N))| $. Hence we may apply Lemma \ref{GA4N} to see that if $A_4$ can be embedded in $G$, then not all groups of order $8$ can be embedded in $G$.

\begin{case}\label{GL(3,2)}

$ |W| = 168 = 2^3.3.7 ; \ W  \cong PSL(3, 2) \cong GL(3, 2) ; \ M \cong C_2 ; \ \bmaut(W)/\bminn(W) \cong C_2 $.

\end{case}

Since $ M \cong C_2 $ we have either $Z(N) = 1 $ or $ Z(N) \cong C_2 $. In addition we see that $ Z(N) = S_1 \cap S_2 $. Since $Z(N)$ is a 2-group, we have $ |P_1| = 3 $. Now $G/C$ is isomorphic to a subgroup of $\bmaut(W)$ and $ 3 \top |\bmaut(W)|$, so $ 3^2 \top |C| $ and $ |P_2| = 9$. We further have $ P_1 \cap P_2 \leqslant N \cap C = Z(N) $. But $ |Z(N)| \leq 2 $, so $ P_1 \cap P_2 = 1 $ and $ |P_1P_2| = |P_1| |P_2| = 27 $. Hence $ P_1P_2 = P_1 \times P_2 $ is a Sylow 3-subgroup of $G$. In particular $NC$ is a normal subgroup of $G$ that contains a Sylow 3-subgroup of $G$ and it follows that $ O^{3^{'}}(G) \leqslant NC $.

We suppose that $A_4$ can be embedded in $G$ and let $ U \leqslant G $ be such that $ U \cong A_4 $. Then $ U \leqslant O^{3^{'}}( G ) \leqslant NC $. If $ U \cap N = 1 $ then
\begin{eqnarray*}
UN/N & \leqslant & CN/N\\
& \cong & C/(C \cap N)\\
& = & C/Z(N)
\end{eqnarray*}
Thus $C/Z(N) $ has a subgroup isomorphic to $U$ and, in particular, has a subgroup isomorphic to $ U^{'  } \cong C_2 \times C_2 $. 

Now $ S_1/Z(N) $ ($ = S_1 / (S_1 \cap S_2 ) $) is isomorphic to a Sylow 2-subgroup of $ W \cong GL(3, 2) $, so $S_1/Z(N) \cong D_4 $. In addition $U^{'} \cong C_2 \times C_2$ can be embedded in $S_2/Z(N)$, so by comparison of orders, $  S_1 \cap S_2 =$ $Z(N) = 1 $  and the Sylow 2-subgroups of $G$ are isomorphic to $ S_1 \times S_2 \cong D_4 \times C_2 \times C_2 $. But then the Sylow 2-subgroups of $G$ have exponent 4, so $C_8$ cannot be embedded in $G$.

We may now assume that $ U \cap N \neq 1 $ so, by minimal normality, $ U^{'} \leqslant N $. Now there exists an element $ x \in U $ such that $ o(x) = 3 $ and $ [x, U^{'}] = U^{'}$. Without loss of generality we assume that $ x \in P_1 \times  P_2 $ and let $ x = x_1x_2 $, with $ x_1 \in P_1 $ and $ x_2 \in P_2 $. Since $P_1$ and $P_2$ commute elementwise we have $ o(x_i) = 1 $ or $ o(x_i) = 3$ $ (i = 1, 2) $. If $ o(x_1) = 1 $ then $ x = x_2 $ so $ U^{'} = [x_2, U^{'}] \leqslant C $. Then $ U^{'} \leqslant C \cap N = Z(N) $, which is a contradiction since $|Z(N)| \leqslant 2$.  Hence $ o(x_1) = 3 $.

Now $ x_2 \in C $ so $x_2$ centralises $N$ and, in particular, $x_2$ centralises $U^{'}$. But $x_2$ also centralises $ x = x_1x_2 $. Thus $x_2$ centralises $ U = \langle \ x_1x_2, \ U^{'} \rangle $. In addition we have $ [ x_1, U^{'} ] = [ x_1x_2, U^{'} ] = U^{'} $, so $x_1$ normalises $U^{'}$. Furthermore, since $x_2$ is a 3-element and $ \langle x_2 \rangle \cap U^{'}\langle x_1 \rangle \leqslant C \cap N = Z(N)$ (which is a 2-group), we have $ \langle x_2 \rangle \cap U^{'}\langle x_1 \rangle = 1 $. Hence
\begin{eqnarray*}
U \langle x_2 \rangle & = & U^{'} \langle \ x_1x_2, \ x_2 \ \rangle\\
& = & U^{'} \langle x_1 \rangle \langle x_2 \rangle\\
& = & U^{'} \langle x_1 \rangle \times \langle x_2 \rangle .
\end{eqnarray*} 
Thus $ U^{'} \langle x_1 \rangle \cong U \langle x_2 \rangle/ \langle x_2 \rangle  \cong U \cong A_4 $, so $A_4$ can be embedded in $N$. Therefore, without loss of generality, we assume that $ U \leqslant N $.

We recall that, since $ M \cong C_2 $, we have either $Z(N) = 1 $ or $ Z(N) \cong C_2 $. We suppose first that $ Z(N) = 1 $. $S_1$ is then isomorphic to a Sylow 2-subgroup of $ GL( 3, 2 ) $, so $ S_1 \cong D_4 $. In particular $S_1$ has exponent 4. In addition, since $ | \bmaut(W) : \bminn(W) | = 2 $, we have $ 2^4 \top |\bmaut(W)|$.  Since $ 2^3 \top |W|$  ($= | N/Z(N) | $) it follows that either $ 2^2 \top |C| $ or $ 2 \top |C| $. If $ 2^2 \top |C| $ then $|S_2| = 4$, so $S_2$ has exponent at most 4. Since $ Z(N) = 1 $, $ S_1S_2 = S_1 \times S_2 $ is a Sylow 2-subgroup of $G$. Hence the Sylow 2-subgroups of $G$ have exponent $4$, so $C_8$ cannot be embedded in $G$.

If $ 2 \top |C| $ then $ S_2 \cong C_2 $ and $ S_1S_2 = S_1 \times S_2 \cong D_4 \times C_2$. Letting $S$ be a Sylow 2-subgroup of $G$ that contains $S_1S_2$, we see that $ S_2 = S \cap C = C_S(N) \unlhd S $. Since $ 2^4 \top | \bmaut( GL( 3, 2 ) ) | $  we see that $S/S_2$ is isomorphic to a Sylow 2-subgroup of $\bmaut( GL( 3, 2 ) )$. Identifying $S/S_2$ with a Sylow 2-subgroup of $ \bmaut( GL( 3, 2 ) ) $ and $ S_1S_2/S_2 $ with a Sylow 2-subgroup of $ \bminn( GL( 3, 2 ) ) $, we apply Corollary \ref{SInny} to see that there exists an element $ y \in S $ such that $ S/S_2 = (S_1S_2/S_2)\langle y \rangle S_2/S_2 $, where $ \langle y \rangle S_2/S_2 \cong C_2 $ and, in particular, $ y^2 \in S_2 $. We may thus apply Corollary \ref{GDxy} to see that not all groups of order 8 can be embedded in $G$.

We may now assume that $ Z(N) \cong C_2 $. We have $ Z(N) \leqslant C $ and, as above, either $ 2^2 \top |C| $ or $ 2 \top |C| $. If $ 2^2 \top |C| $ then $ |S_2| = 4 $, so $S_2$ is abelian. Since $U$, as above, is a subgroup of $N$ we have $ U^{'} \cap C \leqslant Z(U) = 1 $ and $ [ U^{'}, C ] = 1 $. Hence, in particular, we have $ U^{'} \cap S_2 = 1 $ and $ [ U^{'}, S_2 ] = 1 $. Therefore
\[  U^{'}S_2 = U^{'} \times S_2. \]
But $ U^{'} \cong C_2 \times C_2 $ and $ |S_2| = 4 $, so $ U^{'}S_2 $ is abelian of order 16 and has exponent at most $4$. We then see from Lemma \ref{Habex4} that not all groups of order $8$ can be embedded in $G$. 

We now assume that $ Z(N) \cong C_2 $ and that $ 2 \top |C| $. In this case $ Z(N) = S_1 \cap S_2 = S_2 $ is a Sylow 2-subgroup of $C$. Since $ S_2 = Z(N) $, we see in particular that $ S_2 \unlhd G$. As $ U \leqslant N $ we may assume, without loss of generality, that $ U^{'} $ ($\cong C_2 \times C_2 $) is a subgroup of $S_1$. By the minimal normality of $U^{'}$ in $U$ we have $ U \cap S_2 = U^{'} \cap S_2 = 1 $, so $ US_2/S_2 \cong U/( U \cap S_2 ) \cong U \cong A_4 $.

We have $ |U^{'}S_2| = |U^{'}||S_2| = 8 $ and $ |S_1| = 16 $, so $ | S_1 : U^{'}S_2 | = 2 $. It follows that $ U^{'}S_2 \unlhd S_1 $. Thus $ \langle U, S_1 \rangle \leqslant N_N( U^{'}S_2 ) $. Since $ 3 \mid |U| $ we see that $ 3 \times 16 = 48 $ is a divisor of $ |N_N( U^{'}S_2 )| $. Then $ | N/S_2 : N_N( U^{'}S_2)/S_2 | $ is a divisor of $ \frac{168}{24} = 7 $ so, by the simplicity of $N/S_2$, we have $ | N_N( U^{'}S_2 ) | = 48 $. In addition we see, by comparison of orders, that $ N_N( U^{'}S_2 ) = US_1 $. 

Now $ |US_2| = |U||S_2| = 24 $, so $ | US_1 : US_2 | = 2 $ and $ US_2 \unlhd US_1 $. Since $ US_2 = U \times S_2 \cong A_4 \times C_2 $ we see that $ U = O^{3^{'}}( US_2 ) $ is a characteristic subgroup of $ US_2 $, and hence normal in $ US_1 $, and it follows that $ U^{'} $ is also normal in $ US_1 $. We let $ x_1 \in S_1 \char92 U^{'}S_2 $. Then $x_1$ centralises $S_2$ and normalises $U^{'}$. Since $ U^{'}S_2 = U^{'} \times S_2 \cong C_2 \times C_2 \times C_2 $ we see that $ U^{'}S_2 $ is abelian, so
\begin{eqnarray*}
 S_1^{'} & = & [ U^{'}S_2, \langle x_1 \rangle  ]\\
& = & [ U^{'}, \langle x_1 \rangle ] [ S_2, \langle x_1 \rangle ]\\
& = & [ U^{'}, \langle x_1 \rangle ] \leqslant U^{'}.
\end{eqnarray*}
But $ [ U^{'}, \langle x_1 \rangle ] $ is a proper subgroup of $U^{'}$ (since $S_1$ is nilpotent). In addition $ S_1/S_2 $ is isomorphic to a Sylow 2-subgroup of $ GL( 3, 2 ) $, so $S_1/S_2 \cong D_4 $ and $ S_1^{'}S_2/S_2 \cong C_2 $. Hence $ S_1^{'} = [ U^{'}, \langle x_1 \rangle ] \cong C_2 $. We let $ S_1^{'} = \langle x_2 \rangle $ and let $S$ be a Sylow 2-subgroup of $G$ that contains $S_1$. Since $ | S :  S_1 | = 2 $ we have $ S_1 \unlhd S $ and, as $ \langle x_2 \rangle $ is characteristic in $S_1$, we also have $ \langle x_2 \rangle \unlhd S $.

If $ o(g) \geq 4 $ for all $ g \in S_1 \char92 U^{'}S_2 $ then $ U^{'}S_2 = \langle \ x \in S_1 \mid o(x) = 2 \ \rangle $ is a characteristic subgroup of $ S_1 $, and so is normal in $S$. If $Q_2$ can be embedded in $G$ then we let $ Q \leqslant S $ be such that $ Q \cong Q_2 $. Since $ Q_2 $ has a unique element of order 2 we have $ | Q \cap U^{'}S_2 | = 2 $ and $ S = QU^{'}S_2 $. Then
\begin{eqnarray*}
S/U^{'}S_2 & \cong & QU^{'}S_2/ U^{'} S_2\\
& \cong & Q/( Q \cap U^{'}S_2 )\\
& \cong & Q_2/Z(Q_2)\\
& \cong & C_2 \times C_2
\end{eqnarray*}
Thus $ S/U^{'}S_2 $ has exponent 2. Since $ U^{'}S_2$ also has exponent 2 we see that $S$ has exponent at most 4, so $C_8$ cannot be embedded in $G$.

We may now assume that there exists an element $ g \in S_1 \char92 U^{'}S_2 $ such that $ o(g) = 2 $. If $ \langle g \rangle U^{'} \cap S_2 \neq 1 $ then $ S_2 \leqslant \langle g \rangle U^{'} $, whence $ S_1 = \langle g \rangle U^{'}S_2 = \langle g \rangle U^{'} $, and the contradiction $ |S_1| = 8 $ arises. Thus we may assume that $ \langle g \rangle U^{'} \cap S_2 = 1 $, so $ \langle g \rangle U^{'} \cong \langle g \rangle U^{'}S_2/S_2 \cong S_1/S_2 \cong D_4 $. Then, since $ S_2  \leqslant Z(N) $, we have
\[ S_1 = \langle g \rangle U^{'} \times S_2 \cong D_4 \times C_2. \]
Now $S/S_2$ is isomorphic to a Sylow 2-subgroup of $ \bmaut( GL( 3, 2 ) )$ so we may apply Corollary \ref{SInny} to see that there exists an element $ y \in S $ such that $ S/S_2 = ( S_1/S_2 )( \langle y \rangle S_2/S_2 ) $, where $ \langle y \rangle S_2/S_2  \cong C_2 $. In particular we see that $ y^2 \in S_2 $. We can now apply Corollary \ref{GDxy} to see that not all groups of order 8 can be embedded in G. 

This completes the proof for Case \ref{GL(3,2)}.

\begin{case}\label{A6}

$ |W| = 360 = 2^3.3^2.5 ; \ W \cong A_6 ;\  M \cong C_6 ; \ \bmaut(W)/\bminn(W) \cong C_2 \times C_2 $.

\end{case}

If $ 3 \nmid |Z(N)| $ then, since $ 3 \nmid |\bmaut(W)/\bminn(W)| $ and $ 3^2 \mid |W| $, we see that $ 3 \top |C| $. Thus $ P_2 \cong C_3 $. Furthermore, since the Sylow 3-subgroups of $A_6$ are isomorphic to $ C_3 \times C_3 $, we have $ P_1 \cong C_3 \times C_3 $. It follows that $ P_1P_2 = P_1 \times P_2 \cong C_3 \times C_3 \times C_3 $ is a Sylow 3-subgroup of $G$, so $C_9$ cannot be embedded in $G$. 

On the other hand, if $ 3 \mid |Z(N)| $ then we may apply Lemma \ref{C9A6} to see that $C_9$ cannot be embedded in $G$. 

\begin{case}\label{PSL2,8}

$ |W| = 504 = 2^3.3^2.7 ; \  W \cong PSL(2, 8) ; \ M = 1 ; \ \bmaut(W)/\bminn(W) \cong C_3 $.

\end{case}

Since $ M = 1 $ and $ 2^3 \top |W| $ we have $ Z(N) =1 $ and $ |S_1| = 8 $. Now $ 2 \nmid | \bmaut(W)/\bminn(W) | $ so $ |S_2| = 4 $. Then $ S_1S_2 = S_1 \times S_2 $ has order 32, so $S_1S_2$  is a Sylow 2-subgroup of $G$. We let $ K \leqslant S_1 $ be such that $ |K| = 4 $. Then $K$ is abelian of exponent at most $4$ and so $ K \times S_2 $ is a subgroup of $ S_1 \times S_2 $ that is abelian, of order $16$, and which has exponent at most $4$. Hence, by Lemma \ref{Habex4} we see that not all groups of order $8$ can be embedded in $G$.

\begin{case}\label{PSL211}

$ |W| = 660 = 2^2.3.5.11 ; \ W \cong PSL(2, 11) ; \ M \cong C_2 ; \ \bmaut(W)/\bminn(W) \cong C_2 $.

\end{case}

As $ M \cong C_2 $ and $ \bmaut(W)/\bminn(W) \cong C_2 $, we see that the conditions for Lemma \ref{GA4N} are satisfied. Hence if $A_4$ can be embedded in $G$, then not all groups of order $8$ can be embedded in $G$.

\begin{case}

$ |W| = 1092 = 2^2.3.7.13 ; \ W \cong PSL(2, 13) ; \  M \cong C_2 ; \ \bmaut(W)/\bminn(W) \cong C_2 $.

\end{case}

As in Case \ref{PSL211} the conditions for Lemma \ref{GA4N} are satisfied, so if $A_4$ can be embedded in $G$, then not all groups of order $8$ can be embedded in $G$.

\begin{case}

$ |W| = 2520 = 2^3.3^2.5.7 ; \ W \cong A_7 ; \ M \cong C_6 ; \ \bmaut(W)/\bminn(W) \cong C_2 $.

\end{case}

As in Case \ref{A6} we see that $C_9$ cannot be embedded in $G$.

\begin{case}

$ |W| = 5616 = 2^4.3^3.13 ; \ W \cong PSL(3, 3) ; \ M =1 ; \  \bmaut(W)/\bminn(W) \cong C_2$.

\end{case}

By comparison of orders the Sylow 3-subgroups of $G$ are isomorphic to those of $ N/Z(N) \cong PSL(3, 3)$. The Sylow 3-subgroups of $PSL(3, 3)$ are in turn isomorphic to those of $GL( 3, 3 )$, and a Sylow 3-subgroup of $GL( 3, 3 )$ is given by 
$ P = \{ \ \left( \begin{array}{lll}
1 & 0 & 0\\
a & 1 & 0\\
b & c & 1
\end{array} \right) \ \mid \ a, \ b, \ c \in GF(3) \ \} $.
But $P$ is non-abelian and of order 27 and exponent 3, so $C_9$ cannot be embedded in $G$.

\begin{case}

$ |W| = 6048 = 2^5.3^3.7 ; \ W \cong U_3(3) ; \ M = 1 ; \ \bmaut(W)/\bminn(W) \cong C_2 $.

\end{case}

In this case we have $ N \cong W \cong U_3(3) $, so $ 2^5 \top |N|$. It follows that $ O^{2^{'}}(G) \leqslant N $. Since $ 5 \nmid |U_3(3)| $ and $ 5 \nmid | \bmaut(U_3(3)) | $ we see that $C$ contains a Sylow 5-subgroup of $G$, whence $ O^{5^{'}}(G) \leqslant C $. Hence $ O^{2^{'}}(G) \cap O^{5^{'}}(G) \leqslant N \cap C = 1 $, and it follows that $ O^{2^{'}}(G)O^{5^{'}}(G) = O^{2^{'}}(G) \times O^{5^{'}}(G) $. Thus all 2-elements of $G$ commute with all 5-elements of $G$. In particular, all subgroups of order 10 in $G$ are abelian. Hence $D_5$ cannot be embedded in $G$.

\begin{case}

$ |W| = 7920 = 2^4.3^2.5.11 ; \ W \cong M_{11} ; \ M = 1 ; \ \bmaut(W)/\bminn(W) =1 $.

\end{case}

From \cite{Atlas} (p.18) we see that $A_6$ can be embedded in a maximal subgroup of $M_{11}$. As the Sylow 3-subgroups of $A_6$ are isomorphic to $ C_3 \times C_3 $ we see, by comparison of orders, that the Sylow 3-subgroups of $M_{11}$ are also isomorphic to $ C_3 \times C_3$. Since $ M = 1 $ we have $ N \cong M_{11} $. Thus $ P_1 \cong C_3 \times C_3 $. Since $ Z(N) = 1$ we have $ N \cap C = 1 $ and, since $ \bmaut(M_{11}) = \bminn(M_{11})$, we have $ 3 \top |C| $ so $ P_2 \cong C_3$. Hence $ P_1P_2 = P_1 \times P_2 \cong C_3 \times C_3 \times C_3 $ is a Sylow 3-subgroup of $G$ and we see that $C_9$ cannot be embedded in $G$.

\begin{case}

$ |W| = 9828 = 2^2.3^3.7.13 ; \ W \cong PSL(2, 27) ; \ M \cong C_2 ; \ \bmaut(W)/\bminn(W) \cong C_6  $.

\end{case}

$PSL(2, 27)$ is the quotient of a subgroup of $GL(2, 27)$. We see that $ 3^3 \top |GL(2,27)| $. Thus the Sylow 3-subgroups of $ PSL(2, 27) $ are isomorphic to those of $GL(2,27)$. Now
\[ P = \{ \ \left( \begin{array}{ll}
1 & 0\\
k & 1
\end{array} \right) \ \mid \ k \in GF(27) \ \}  
\cong GF(27)^{+}  
\]
is a Sylow 3-subgroup of $ GL(2,27) $. As $ GF(27)^{+} $ has characteristic 3, $P$ is elementary abelian. Thus the Sylow 3-subgroups of $G$ are elementary abelian, so $C_9$ cannot be embedded in $G$. 

This concludes the proof of the proposition.

\end{proof}

\section{Minimal embeddings of groups of order $n$ or less for $n \leq 15$}\label{leqnleq15}

We finally construct examples to show that the bound given by Lemma \ref{nbound} can be attained for $ n = 1, \dots , 11$, and that the bounds given by Corollary \ref{boundnleq15} are attained in both the soluble and non-soluble cases for $ n = 12, \dots , 15 $. In contrast to Theorem \ref{embednleq15}, a complete list of the respective minimal groups is not given. Using Theorem \ref{embednleq15}, it can be verified that groups in the following table are indeed minimal with respect to the embedding of all groups of order $n$ or less for $ 1 \leq n \leq 11$. The table thus shows that the bound given by Lemma \ref{nbound} can be achieved by more than one group for $ n = 3, \dots , 11 $. We recall from Section \ref{order8} that $H_1$ refers to the quasi-dihedral group of order 16.

\begin{table}[H]

\begin{center}

\begin{tabular}{cc}
\hline
$n$&\begin{tabular}{@{}c@{}}Examples of groups of minimal order in which all groups\\[-3pt]of order $n$ or less can be embedded\\ \end{tabular}\\
\hline
\rule{0pt}{3ex}

$1$ & $C_{1}$\\ 
$2$ & $C_{2}$\\ 
$3$ & $C_{6}, \ D_{3}$\\
$4$ & $C_{4} \times D_{3}, \ S_{4}$\\  
$5$ & $C_{4} \times C_{5} \times D_{3}, \ S_{5}$\\
$6$ & $C_{4} \times C_{5} \times D_{3}, \ S_{5}$\\  
$7$ & $C_{4} \times C_{5} \times C_{7} \times D_{3}, \ C_{7} \times  S_{5}$\\  
$8$ & $C_{5} \times C_{7} \times D_{3} \times H_1, \ C_{7} \times D_{15} \times H_1 $\\ 
$9$ & $C_{5} \times C_{7} \times C_{9} \times D_{3} \times H_1, \ C_{7} \times C_{9} \times D_{15} \times H_1 $\\
$10$ & $C_{7} \times C_{9} \times D_{15} \times H_1, \ C_{9} \times D_{105} \times H_1 $\\ 
$11$ & $C_{7} \times C_{9} \times C_{11}  \times D_{15} \times H_1, \ C_{9}  \times C_{11} \times D_{105} \times H_1 $\\ 
\hline
\end{tabular}

\end{center}

\caption{Some groups of minimal order in which all groups of order n or less can be embedded, where $ 1 \leq n \leq 11$}\label{tab1leqnleq11}

\end{table}
Table \ref{tab1leqnleq11} shows that, for $ n = $ 5, 6 and 7, both soluble and insoluble minimal groups can be constructed, though it is not known at present whether this will always be the case for $ n < 60 $. We now come to the construction of examples of groups of minimal order in which all groups of order $n$ can be embedded for $ 12 \leq n \leq 15 $. The groups in question are subgroups of the direct product of extensions of certain groups of small order by suitable transpositions. We first define the groups $ K_1, \dots , K_7 $ as follows:
\begin{eqnarray*}
K_1 & = & \langle \ a_1, \ a_2 \mid a_1^4 = a_2^4 = 1, \  a_1^2 =  a_2^2, \ a_1^{-1}a_2a_1 = a_2^3  \ \rangle \cong Q_2 ;\\
K_2 & = & \langle \ b_1, \ b_2,\  b_3 \mid b_1^2 = b_2^2 = b_3^3 = 1, \ [ b_1, b_2 ] = 1, \ b_3^{-1} b_1 b_3 = b_2, \ b_3^{-1} b_2 b_3 = b_1b_2 \ \rangle \cong A_4 ;\\
K_3 & = & \langle c \rangle \cong C_9 ;\\
K_4 & = & \langle d \rangle \cong C_5 ;\\
K_5 & = & \langle e \rangle \cong C_7 ; \\
K_6 & = & \langle f \rangle \cong C_{11} ; \\
K_7 & = & \langle g \rangle \cong C_{13} . \\
\end{eqnarray*}
We define an automorphism $ \theta_{i} $,  of order 2, for each  group $K_{i}$  by
\begin{eqnarray*}
&& a_1^{\theta_{1}}  =  a_2, \  a_2^{\theta_{1}}  =  a_1 ;\\
&& b_1^{\theta_{2}}  =  b_2 , \  b_2^{\theta_{2}}  =  b_1 ,\  b_3^{\theta_{2}} = b_3^{-1} ;\\
&& c^{\theta_{3}} = c^{-1} ;\\
&& d^{\theta_{4}} = d^{-1} ;\\
&& e^{\theta_{5}} = e^{-1} ;\\
&& f^{\theta_{6}} = f^{-1} ;\\
&& g^{\theta_{7}} = g^{-1} .\\
\end{eqnarray*}
We extend $K_i$ by $ \langle \theta_{i} \rangle $ to form the semi-direct product $ K_i \langle  \theta_i \rangle $ for $ i = 1, \dots , 7 $. Expressing the symmetric group $ S_5 = A_5 \langle (12) \rangle $ in cycle notation, we then form the direct product:
\[ G = K_1 \langle \theta_1 \rangle \times ... \times K_7 \langle \theta_7 \rangle \times A_5 \langle (12) \rangle .\]
We note that $G$ can be considered as a faithful extension of the direct product 
\[ Q_2 \times A_4 \times C_9 \times C_5 \times C_7 \times C_{11} \times C_{13} \times A_5 \]
by the elementary abelian 2-group of rank 8, $\langle \ \theta_{1}, \dots , \theta_{7}, \ (12) \ \rangle $.

The groups of order 15 or less can be embedded in $G$ in various ways. In particular the following isomorphisms can be verified:
\begin{eqnarray*}
&& D_3 \cong \langle \ c^3, \ \theta_{3}  \theta \ \rangle , \ \theta \in \langle \ \theta_{1}, \ \theta_{2}, \ \theta_{4}, \dots, \theta_{7}, \ (12) \ \rangle ;\\
&& D_3 \cong \ \langle \ (123), \ (12) \theta \ \rangle , \  \theta \in \langle \ \theta_{1},  \dots, \theta_{7} \ \rangle ;\\
&& C_8 \cong \langle \ a_1 \theta_1 \theta \ \rangle  , \ \theta \in \langle \ \theta_2, \dots , \theta_7 , \  (12) \ \rangle ;\\
&& C_2 \times C_4 \cong \ \langle \ a_1, \ b_1 \ \rangle \cong \langle \ a_1, \ (12)(34) \ \rangle ;\\
&& C_2 \times C_2 \times C_2 \cong \langle \ a_1^2, \ b_1, \ b_2 \ \rangle \cong \langle \ a_1^2, \ (12)(34), \ (14)(32) \ \rangle ;\\
&& C_2 \times C_2 \times C_2 \cong \langle \ a_1^2, \ (12)(34), \ \theta \ \rangle , \ 1 \neq \theta \in \langle \ \theta_{1},  \dots, \theta_{7} \ \rangle ;\\
&& D_4 \cong \langle \ a_1a_2, \ \theta_1 \theta \ \rangle , \ \theta \in \langle \ \theta_2 , \dots \theta _7 , \ (12) \ \rangle ;\\
&& C_3 \times C_3 \cong \langle \ b_3 , \ c^3 \ \rangle \cong \langle \ c^3 , \ (123) \ \rangle ;\\
&& C_{10} \cong \langle \ a_1^2 d \ \rangle \cong \langle \ a_1^2 (12345) \ \rangle ;\\
&& D_5 \cong \langle \ (12345), \ (15)(24) \ \rangle \cong \langle\  d, \ \theta_4 \theta \ \rangle , \ \theta \in \langle \ \theta_1, \ \theta_2, \ \theta_3, \ \theta_5 , \ \theta_6, \ \theta_7 \ \rangle ;\\
&& C_2 \times C_2 \times C_3 \cong C_2 \times C_6 \cong \langle \ b_1, \ b_2c^3 \ \rangle \cong \langle \ (12)(34), \  (14)(32)c^3 \  \rangle ;\\
&& C_{12} \cong \langle \ a_1c^3 \ \rangle ;\\
&& D_6 \cong \langle \ a_1^2b_3, \ \theta_2 \theta \ \rangle , \ \theta \in \langle \ \theta_1, \ \theta_3, \dots , \theta_7 \ \rangle ;\\
&& D_6 \cong \langle \ a_1^2(123), \ (12) \theta \ \rangle , \ \theta \in \langle \ \theta_1, \dots , \theta_7 \ \rangle ;\\
&& Q_3 \cong  \langle \ b_1  \theta_2 \theta_3 \theta, \ c^3 \ \rangle , \ \theta \in \langle \ \theta_1, \  \theta_4 , \dots, \ \theta_7 \ \rangle ; \\
&& Q_3 \cong \langle \ (14)(32)(12) \theta_3 \theta, \ c^3 \ \rangle , \ \theta \in \langle \ \theta_1, \ \theta_2, \ \theta_4, \dots, \theta_7 \ \rangle ;\\
&& C_{14} \cong \langle \ a_1^2e \ \rangle ;\\
&& D_7 \cong \langle \ e, \ \theta_5 \theta \ \rangle , \ \theta \in \langle \ \theta_1, \dots , \theta _4 , \ \theta_6, \ \theta_7, \ (12) \ \rangle ;\\
&& C_{15} \cong \langle \ c^3d \ \rangle \cong \langle \ c^3(12345) \ \rangle .\\
\end{eqnarray*}

Bearing in mind the above embeddings, Theorem \ref{embednleq15}, and the isomorphism types of $ K_1, \dots , K_7 $, it can be verified that the following families of subgroups of $G$ will attain the bounds given by Corollary \ref{boundnleq15} for the respective values of $n$. The groups are presented as extensions of subgroups of $ K_1 \times \dots \times K_7 \times A_5 $, referred to as ``base groups ", by suitable subgroups of order 2 in $ \langle \ \theta_1, \dots, \theta_7, \ (12) \ \rangle $. The order of the extension and the isomorphism type of the base group is given in each case.

\begin{tabular}{ll}

$ n = 12 $ (soluble) & \\
Extensions : & $ ( K_1 \times K_2 \times K_3 \times K_4 \times K_5 \times K_6 ) \langle \ \theta_1 \theta_2 \theta_3 \theta_4 \theta \  \rangle$, $\theta \in \langle \ \theta_5 , \ \theta_6 \  \rangle$\\
Order : & $ 2^6.3^3.5.7.11 = 665 280$\\
Base Group : &  $ Q_2 \times A_4 \times C_9 \times C_5 \times C_7 \times C_{11} $\\
 & \\
$ n = 12 $ (non-soluble) & \\
Extensions : & $ ( K_1 \times K_3 \times K_5 \times K_6 \times A_5 ) \langle \ \theta_1 \theta_3 (12) \theta \ \rangle$, $\theta \in \langle \   \theta_5 , \ \theta_6 \  \rangle$\\
Order : & $ 2^6.3^3.5.7.11 = 665 280$\\
Base Group : &  $ Q_2 \times C_9 \times C_7 \times C_{11} \times A_5 $\\
 & \\
$ n = 13, 14, 15 $ (soluble) & \\
Extensions : & $ ( K_1 \times K_2 \times K_3 \times K_4 \times K_5 \times K_6 \times K_7 ) \langle \ \theta_1 \theta_2 \theta_3 \theta_4 \theta_5 \theta \ \rangle$, $\theta \in \langle \ \theta_6 , \ \theta_7  \ \rangle$\\
Order : & $ 2^6.3^3.5.7.11.13 = 8 648 640$\\
Base Group : &  $ Q_2 \times A_4 \times C_9 \times C_5 \times C_7 \times C_{11} \times C_{13}$\\
 & \\
$ n = 13, 14, 15 $ (non-soluble) & \\
Extensions : & $ ( K_1 \times K_3 \times K_5 \times K_6 \times K_7 \times A_5 ) \langle \ \theta_1 \theta_3 \theta_5 (12) \theta \ \rangle$, $\theta \in \langle \ \theta_6 , \ \theta_7  \ \rangle$\\
Order : & $ 2^6.3^3.5.7.11.13 = 8 648 640$\\
Base Group : &  $ Q_2 \times C_9 \times C_7 \times C_{11} \times C_{13} \times A_5$\\
 & \\
\end{tabular}

These examples show that, for $ n = 12, \dots , 15 $, the bound given in Corollary \ref{boundnleq15} is attained by both soluble and insoluble groups. Proposition \ref{12embed} shows that the orders of the respective groups are also minimal in the non-soluble case. Our final result summarises our findings as regards the minimal order of a group in which all groups of order $n$ or less can be embedded, for  $ 1 \leq n \leq 15 $.

\begin{proposition}\label{generalcase}

Let $G$ be a group of minimal order in which all groups of order $n$ or less can be embedded. Then, for $ 3 \leq n \leq 15 $, $G$ is not unique and $|G|$ is as in the following table:

\begin{table}[H]
\begin{center}
\begin{tabular}{ccccc}
\hline
$n$&Minimal order && $n$&Minimal order\\
\hline
\rule{0pt}{3ex}
$1$ & $1$ &\ \ & $9$ & $ 30 \ 240 = 2^{5} \cdot  3^{3} \cdot 5 \cdot 7 $\\ 
$2$ & $2$ && $10$ & $ 30 \ 240 $\\ 
$3$ & $6 = 2\cdot 3 $ && $11$ & $ 332 \ 640 = 2^{5} \cdot 3^{3} \cdot 5 \cdot 7 \cdot 11 $\\
$4$ & $ 24 = 2^{3} \cdot 3 $ && $12$ & $ 665 \ 280 = 2^{6} \cdot 3^{3} \cdot 5 \cdot 7 \cdot 11 $\\
$5$ & $ 120 = 2^{3} \cdot 3 \cdot 5 $ && $13$ & $ 8 \ 648 \ 640 = 2^{6} \cdot 3^{3} \cdot 5 \cdot 7 \cdot 11 \cdot ^13 $\\
$6$ & $ 120 $ && $14$ & $ 8 \ 648 \ 640 $\\
$7$ & $ 840 = 2^{3} \cdot 3 \cdot 5 \cdot 7 $ && $15$ & $ 8 \ 648 \  640 $\\
$8$ & $ 3 \ 360 = 2^{5} \cdot 3 \cdot 5 \cdot 7 $ && &\\
\hline
\end{tabular}
\end{center}
\caption{Minimal order of a group in which all groups of order $n$ or less can be embedded, for $ 1 \le n \le 15 $}\label{MinGroups3}
\end{table}

\end{proposition}

\begin{flushleft}

\vspace{2 mm}

Robert Heffernan  \\
University of Connecticut,\\
Storrs,\\
CT 06269-3009\\
USA\\
e-mail: robert.heffernan@uconn.edu

\vspace{6 mm}

Des MacHale \\
University College, Cork\\
Ireland\\
d.machale@ucc.ie

\vspace{6 mm}

Brendan McCann\\
Department of Mathematics and Computing,\\
Waterford Institute of Technology,\\
Cork Road,\\
Waterford,\\
Ireland\\
e-mail: bmccann@wit.ie

\end{flushleft}

\end{document}